\documentclass{amsart}
\usepackage{amssymb}
\usepackage{amsthm}
\usepackage{xcolor}
\usepackage{tikz}
\usepackage{amsaddr}
\newtheorem{Theorem}{Theorem}[section]
\newtheorem{Lemma}[Theorem]{Lemma}

\newtheorem{Corollary}[Theorem]{Corollary}
\newtheorem{Observation}[Theorem]{Observation}
\theoremstyle{remark} \newtheorem*{sketch}{Sketch of Proof}
\theoremstyle{remark} 
\theoremstyle{remark} 
\newtheorem*{appsketch}{Sketch of Proof of Lemma~\ref{mainapps}}
\theoremstyle{remark} \newtheorem*{appfull}{Full Proof of Lemma~\ref{mainapps}}
\theoremstyle{remark} 
\newtheorem*{6sketch}{Sketch of Proof of Lemma~\ref{main6apps}}
\theoremstyle{remark} \newtheorem*{6full}{Full Proof of Lemma~\ref{main6apps}}
\theoremstyle{remark} \newtheorem*{cptproof}{Proof of Theorem~\ref{maincpt}}
\theoremstyle{plain} 
\theoremstyle{plain} 
\theoremstyle{plain} 
\theoremstyle{definition} \newtheorem{Definition}[Theorem]{Definition}
\allowdisplaybreaks
\begin{document}
\title{The evolution of random graphs on surfaces}
\author{Chris Dowden, Mihyun Kang, and Philipp Spr\"ussel}
\thanks{The authors are supported by
Austrian Science Fund (FWF): P27290.}
\thanks{An extended abstract of this paper has been published in the proceedings of the European Conference on Combinatorics, Graph Theory and Applications (EuroComb17), Electronic Notes in Discrete Mathematics 61 (2017), 367--373.}
\address{
Graz University of Technology,
Institute of Discrete Mathematics,
Steyrergasse 30,
8010 Graz, 
Austria}
\author{}
\email{\{dowden,kang,spruessel\}@math.tugraz.at} 
\setlength{\unitlength}{1cm}

\begin{abstract}
For integers $g,m \geq 0$ and $n>0$,
let $S_{g}(n,m)$ denote the graph taken uniformly at random
from the set of all graphs on $\{1,2, \ldots, n\}$
with exactly $m=m(n)$ edges and with genus at most $g$.
We use counting arguments to investigate the components, subgraphs,
maximum degree, and largest face size of $S_{g}(n,m)$,
finding that there is often different asymptotic behaviour 
depending on the ratio $\frac{m}{n}$.

In our main results,
we show that the probability that $S_{g}(n,m)$
contains any given non-planar component converges to $0$
as $n \to \infty$ for all $m(n)$;
the probability that $S_{g}(n,m)$ contains a copy of any given planar graph
converges to $1$ as $n \to \infty$ 
if $\liminf \frac{m}{n} > 1$;
the maximum degree of $S_{g}(n,m)$ is $\Theta (\ln n)$
with high probability if $\liminf \frac{m}{n} > 1$;
and the largest face size of $S_{g}(n,m)$ has a threshold around $\frac{m}{n}=1$
where it changes from $\Theta (n)$ to $\Theta (\ln n)$
with high probability.
\end{abstract}
\maketitle

\section{Introduction}

\subsection{Background and motivation} \label{background}

Random planar graphs have been the subject of much activity,
and many properties of the standard random planar graph $P(n)$
(taken uniformly at random from the set of all planar graphs
with vertex set $[n] := \{1,2, \ldots, n\}$) are now known.
For example,
asymptotic results have been obtained for the probability
that $P(n)$ will contain given components and subgraphs
\cite{gim,msw},
for the number of vertices of given degree
\cite{drm2},
and for the size of the maximum degree and largest face
\cite{drm,mcdr}.
In addition,
clever algorithms for generating and sampling planar graphs
have been designed
\cite{bri,fusy},
and random planar maps have been studied
\cite{drm3,gao1,gao2}.

The classical Erd\H{o}s-R\'enyi random graph $G(n,m)$
is taken uniformly at random from the set of all graphs
on $[n]$ with exactly $m$ edges.
Hence,
it is natural to also examine the planar analogue $P(n,m)$,
taken uniformly at random from the set of all planar graphs
on $[n]$ with exactly $m$ edges.
Note that the extra condition on the number of edges
typically makes $P(n,m)$ more challenging to study than $P(n)$,
but many exciting results have nevertheless been obtained
\cite{ben,bod,chap2,dow,dowejc,ger,gim,kangluc}.

Although the constraint on the number of edges
makes $P(n,m)$ more difficult to analyse,
it also has the effect of producing richer and more complex behaviour.
In particular, 
results are often found to feature thresholds,
meaning that the various probabilities change dramatically 
according to which `region' the ratio $\frac{m}{n}$ falls into.

It is well known that $P(n,m)$
behaves in the same way as $G(n,m)$
if $m < \frac{n}{2} - \omega (n^{2/3})$,
since the probability that $G(n,m)$ will be planar 
converges to $1$ as $n \to \infty$ for this range of $m$
(see, for example,~\cite{janluc}).
However,
different properties have been found to emerge when we are beyond this region
\cite{dow,dowejc,ger,gim,kangluc}.

In this paper,
we shall be interested in graphs with genus at most $g$.
A graph is said to have genus at most $g$
if it can be embedded without any crossing edges
on an orientable surface of genus $g$
(i.e.~a sphere to which $g$ handles have been attached).
Hence,
the simplest case when $g=0$ corresponds to planar graphs.

We shall let $S_{g}(n)$
denote the graph taken uniformly at random from the set of all graphs
on $[n]$ with genus at most $g$,
and we shall let $S_{g}(n,m)$
denote the graph taken uniformly at random from the set of all graphs
on $[n]$ with exactly $m$ edges
and with genus at most $g$
(it is known that this then implies that we must have $m \leq 3n-6+6g$).
Throughout the paper,
$m=m(n)$ will be a function of $n$,
while $g$ will be a constant independent of $n$.

Many of the results on $P(n)$ (i.e.~$S_{0}(n)$)
have now also been generalised to $S_{g}(n)$
\cite{chap,mcd,mcdr}.
However,
similar extensions have not yet been achieved for the full $S_{g}(n,m)$ case,
where the complexity of the general genus setting
is combined with the extra restriction on the number of edges.
As with the planar case,
one might expect to find even more interesting behaviour 
for $S_{g}(n,m)$ than $S_{g}(n)$,
and so it is the random graph $S_{g}(n,m)$
that is to be the subject of this paper.

We shall investigate the probability
that $S_{g}(n,m)$ will contain given components and subgraphs,
as well as the size of the maximum degree and largest face
(maximised over all embeddings with genus at most $g$).
We shall find that the restriction on the number of edges 
does indeed enrich the results,
by providing different behaviour depending
on the ratio $\frac{m}{n}$.
Hence, 
this change as $\frac{m}{n}$ varies
can be thought of as the `evolution' of random graphs on surfaces.

\subsection{Main results} \label{results}

We shall now identify the main results of the paper
(the proofs of which will be given later).

One of the most important aspects to consider
when building an understanding of how a random graph behaves
is to gain knowledge of the typical components and subgraphs.
For the $g=0$ case,
it is known 
that the probability that $P(n,m)$ will have any given planar component
is bounded away from $0$
if $\frac{m}{n}$ is in the region 
$1 < \liminf \frac{m}{n} \leq \limsup \frac{m}{n} < 3$
(see~\cite{dowejc}).
However,
for the general genus case,
we surprisingly discover that
the probability that $S_{g}(n,m)$ 
will have any given \emph{non-planar} component
actually converges to $0$ as $n \to \infty$
for \emph{every} function $m(n)$,
even for a component with genus at most $g$:

\begin{Theorem} \label{maincpt}
Let $H$ be a (fixed) connected non-planar graph,
let $g \geq 0$ be a constant,
and let $m=m(n)$ satisfy $m \leq 3n-6+6g$ for all $n$.
Then
\begin{displaymath}
\mathbb{P}
[S_{g}(n,m) \textrm{ will have a component isomorphic to } H] \to 0
\textrm{ as } n \to \infty.
\end{displaymath}
\end{Theorem}

One of the key tools in our proofs will be the use 
of `appearances' and `triangulated appearances',
through which we are able to derive many interesting results.
One of the most fundamental is to show that the probability 
that $S_{g}(n,m)$ will contain a copy of any given \emph{planar} subgraph
converges to $1$ as $n \to \infty$
as long as $\liminf \frac{m}{n} > 1$:

\begin{Theorem} \label{mainsub}
Let $H$ be a (fixed) connected planar graph,
let $g \geq 0$ be a constant,
and let $m=m(n)$ satisfy
$\liminf \frac{m}{n} > 1$
and $m \leq 3n-6+6g$ for all $n$.
Then
\begin{displaymath}
\mathbb{P}
[S_{g}(n,m) \textrm{ will have a copy of } H] \to 1
\textrm{ as } n \to \infty.
\end{displaymath}
\end{Theorem}

The topic of maximum degree
(which we denote throughout by $\Delta$)
is one of the core areas in the study of random graphs,
both as a matter of great interest in its own right
and for its connections to colouring algorithms.
However,
this issue has not previously been explored even for the $g=0$ case 
of the random planar graph $P(n,m)$.
In this paper,
we derive bounds for the full general genus case,
observing that intriguingly there is different behaviour
depending on the ratio of $m$ to $n$
(see Definition~\ref{orderdef} 
for details of the terminology and notation used):

\begin{Theorem} \label{mainmaxdeg}
Let $g \geq 0$ be a constant,
and let $m=m(n)$ satisfy $m \leq 3n-6+6g$ for all $n$.
Then with high probability\footnote{meaning 
with probability tending to $1$ as $n \to \infty$ ---
see Definition~\ref{orderdef}.}
\begin{displaymath}
\Delta (S_{g}(n,m)) =
\left\{ 
{\renewcommand{\arraystretch}{1.5}
\begin{array}{lll}
\Theta \left( \frac{\ln n}{\ln \ln n} \right)
& \textrm{for }
0 < \liminf \frac{m}{n} \leq \limsup \frac{m}{n} < \frac{1}{2}, \\
O(\ln n) 
& \textrm{for } 
\frac{1}{2} \leq \liminf \frac{m}{n} \leq \limsup \frac{m}{n} \leq 1, \\
\Theta (\ln n) 
& \textrm{for }
1 < \liminf \frac{m}{n}.
\end{array}} \right. 
\end{displaymath}
\end{Theorem}

The size of the largest face is another exciting topic
that has not previously been investigated
for $P(n,m)$.
We let $F(S_{g}(n,m))$
denote the size of the largest face of $S_{g}(n,m)$
(maximised over all embeddings with genus at most $g$),
and we again deal with the full general genus case,
finding that there is an interesting threshold around $\frac{m}{n}=1$:

\begin{Theorem} \label{mainface}
Let $g \geq 0$ be a constant,
and let $m=m(n)$ satisfy $m \leq 3n-6+6g$ for all $n$.
Then with high probability
\begin{displaymath}
F(S_{g}(n,m)) =
\left\{ 
{\renewcommand{\arraystretch}{1.5}
\begin{array}{ll}
\Theta (n)
& \textrm{for }
0 < \liminf \frac{m}{n} \leq \limsup \frac{m}{n} < 1, \\
\Theta (\ln n) 
& \textrm{for }
1 < \liminf \frac{m}{n} \leq \limsup \frac{m}{n} < 3.
\end{array}} \right. 
\end{displaymath} 
\end{Theorem}

\subsection{Techniques and outline of the paper} \label{techniques}

Many of our proofs rely on the technique of double-counting.
For example,
suppose that we wish to relate the number of graphs in two sets 
$\mathcal{G}_{n}$ and $\mathcal{G}_{n}^{\prime}$
(e.g.~$\mathcal{G}_{n}^{\prime}$
may be the set of all graphs on $[n]$
with exactly $m$ edges and genus at most $g$,
and $\mathcal{G}_{n}$
may be the subset of $\mathcal{G}_{n}^{\prime}$
consisting of those graphs 
that have a particular property).
For each graph in $\mathcal{G}_{n}$,
we would aim to construct many graphs in $\mathcal{G}_{n}^{\prime}$
by making various alterations
(e.g.~adding/deleting edges in suitable places),
and we would then try to show that each graph in $\mathcal{G}_{n}^{\prime}$
is not constructed too many times.

If, say, 
each graph in $\mathcal{G}_{n}$
can be used to construct $f_{1}(n) = \Omega(n)$ 
graphs in $\mathcal{G}_{n}^{\prime}$,
and each graph in $\mathcal{G}_{n}^{\prime}$
is only constructed $f_{2}(n) = o(n)$ times in total,
then this would imply that
$|\mathcal{G}_{n}^{\prime}| = \frac{\Omega (n) |\mathcal{G}_{n}|}{o(n)}$,
and so 
$\frac{|\mathcal{G}_{n}|}{|\mathcal{G}_{n}^{\prime}|} 
= \frac{o(n)}{\Omega(n)} \to 0$ 
as $n \to \infty$.
Hence,
if $\mathcal{G}_{n}$ is the set of graphs in $\mathcal{G}_{n}^{\prime}$
with a particular property,
then we would conclude that the probability that a random graph in
$\mathcal{G}_{n}^{\prime}$
has this property must converge to $0$ as $n \to \infty$.

The challenge when creating such a proof 
lies in finding a successful way to construct many graphs of the desired type 
without introducing a large amount of double-counting.
Hence,
the alterations used in the construction process
need to be carefully controlled,
in order to allow some way of bounding the number of possibilities 
for the original graph.

Our restriction on the total number of edges in the graph
provides a further serious complication,
as any inserted edges need to be exactly balanced by deleted edges
without losing control of the double-counting.
For graphs on surfaces,
an additional major difficulty arises when inserting edges,
as it is crucial to ensure that any conditions on the genus are not violated.

Thus, 
to obtain workable double-counting arguments,
we shall typically require information about the number of choices we have
for where to add/delete certain types of edges.
In particular,
we shall find that
it will often be useful to have many ways 
(e.g.~linear in the number of vertices)
to insert an edge 
without increasing the genus of the graph,
or to know that there are many `pendant' edges
or many `appearances' of certain subgraphs
(see Definitions~\ref{pendantdefn} and~\ref{defapps}).

The paper is consequently structured as follows:
in Section~\ref{prelim},
we provide the necessary definitions and state the various key lemmas
that will be used during our counting arguments;
in Section~\ref{cptsect},
we collect together results 
on the probability that $S_{g}(n,m)$
will contain given components (including Theorem~\ref{maincpt});
in Section~\ref{subsect},
we will do likewise for subgraphs (including Theorem~\ref{mainsub});
in Section~\ref{maxdegsect},
we will look at the maximum degree of $S_{g}(n,m)$
(obtaining Theorem~\ref{mainmaxdeg});
in Section~\ref{facesize},
we shall investigate the size of the largest face of $S_{g}(n,m)$
(obtaining Theorem~\ref{mainface});
in Sections~\ref{appsection} and~\ref{6appsection},
we will prove two important results from Section~\ref{prelim}
on `appearances' (Lemma~\ref{mainapps})
and `triangulated appearances' (Lemma~\ref{main6apps});
and then in Section~\ref{discussion},
we shall discuss various questions that remain unanswered.

Let us note that some existing results for the planar case $P(n,m)$ can immediately be carried over to $S_{g}(n,m)$.
In particular, this is true
if the relevant proofs only utilise basic properties such as being able to delete an edge or insert an edge between components without increasing the genus.
When we meet such cases,
we shall consequently just state our results without repeating full details of the proofs.

\section{Preliminaries} \label{prelim}

In this section,
we shall provide details of some preliminary matters
that will be of importance to us later.
We begin (in Subsection~\ref{notation})
by stating the notation and definitions that will be used,
and then (in Subsection~\ref{lemmas})
we introduce various lemmas 
that will be essential for many of our proofs.

\subsection{Notation and definitions} \label{notation}

Throughout this paper,
we shall always take $g$, $n$ and $m=m(n)$
to be integers satisfying $g \geq 0$, $n>0$
and $0 \leq m \leq 3n-6+6g$,
even if this is not always explicitly stated.

We start with the notation for our random graph:

\begin{Definition}
We shall let 
\emph{$\mathcal{S}^{g}(n,m)$}
denote the set of all labelled graphs 
on the vertex set $[n] := \{1,2, \ldots, n\}$
with exactly $m$ edges and with genus at most $g$,
and we shall let 
\emph{$S_{g}(n,m)$}
denote a graph taken uniformly at random from $\mathcal{S}^{g}(n,m)$.
\end{Definition}

Next,
we provide details of the order notation 
that will be used throughout this paper:

\begin{Definition} \label{orderdef}
Given non-negative functions $f(n)$ and $h(n)$, 
we shall use the following notation:
\begin{itemize}
\item
$f(n) = \Omega (h(n))$ \\
means there exists a constant $c>0$
such that $f(n) \geq c h(n)$ for all large $n$;
\item $f(n) = O(h(n))$ \\
means there exists a constant $C$
such that $f(n) \leq C h(n)$ for all large $n$;
\item $f(n) = \omega (h(n))$ \\
means $\frac{f(n)}{h(n)} \to \infty$ as $n \to \infty$;
\item $f(n) = o(h(n))$ \\
means $\frac{f(n)}{h(n)} \to 0$ as $n \to \infty$.
\end{itemize}

We shall say that a random event $X_{n}$ happens
\emph{with high probability (whp)}
if
$\mathbb{P}(X_{n}) \to 1$ as $n \to \infty$.
Given a non-negative random variable $f(n)$ 
and a non-negative function $h(n)$,
we shall use the following notation:
\begin{itemize}
\item $f(n) = \Omega (h(n))$ whp \\
means there exists a constant $c>0$
such that $f(n) \geq c h(n)$ whp;
\item $f(n) = O(h(n))$ whp \\
means there exists a constant $C$
such that $f(n) \leq C h(n)$ whp;
\item $f(n) = \omega (h(n))$ whp \\
means that,
given any constant $K$,
we have
$\frac{f(n)}{h(n)} > K$ whp;
\item $f(n) = o(h(n))$ whp \\
means that,
given any constant $\epsilon > 0$,
we have
$\frac{f(n)}{h(n)} < \epsilon$ whp.
\end{itemize}

Throughout this paper,
we shall always take all asymptotics to be as $n \to \infty$,
even if this is not always explicitly stated.
\end{Definition}

We shall often obtain different results
for different types of subgraphs,
and consequently we shall find it convenient to use the following definitions:

\begin{Definition}
We shall say that a connected graph $H$ is \emph{unicyclic}
if $e(H)=|H|$,
and \emph{multicyclic}
(also known as `complex',
see e.g.~\cite{janluc,kangluc})
if $e(H)>|H|$.
\end{Definition}

Many of our counting arguments
will rely heavily on the use of `pendant' edges and copies:

\begin{Definition} \label{pendantdefn}
We shall use \emph{pendant vertex}
to mean a vertex of degree $1$,
and \emph{pendant edge}
to mean an edge incident to such a vertex.

Given a (small) graph $H$ and a (large) graph $G$,
we shall use the term \emph{copy} of $H$
to mean any subgraph of $G$ isomorphic to $H$.
If $H$ is connected,
we shall use \emph{pendant copy} of $H$
to mean an induced copy of $H$
that is joined to the rest of $G$ by exactly one edge.
\end{Definition}

Since a graph $H$ can only have at most $|H|-1$ cut-edges,
we may make the following useful observation:

\begin{Observation} \label{appinter}
A pendant copy of a connected graph $H$
can only have a vertex in common 
with at most $|H|-1$ other pendant copies of $H$.
\end{Observation}

One particular type of pendant copy
that will be extremely important to us is an `appearance':

\begin{Definition} \label{defapps}
Let $H$ be a connected graph on the vertex set $[|H|]$,
and let $G$ be a graph on the vertex set $[n]$,
where $n>|H|$.
Let $W \subset V(G)$ with $|W|=|H|$,
and let the `root' $r_{W}$ denote the smallest element in $W$.
We say that $H$ \emph{appears} at $W$ in $G$ if
(a) the increasing bijection from $[|H|]$ to $W$ gives
an isomorphism between $H$ and the induced subgraph $G[W]$ of $G$;
and (b) there is exactly one edge in $G$ between $W$ and the rest of $G$,
and this edge is incident with the root $r_{W}$.
See Figure~\ref{appfig}
(and note that,
in the particular example shown,
any non-trivial permutation of the labels $\{2,4,5,7\}$
would violate the definition of an appearance).

We say that two appearances at $W_{1}$ and $W_{2}$
are \emph{vertex-disjoint} if $W_{1} \cap W_{2} =~\emptyset$.
\end{Definition}

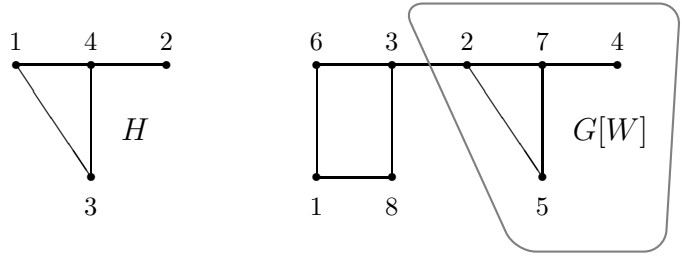
\begin{figure} [ht]
\setlength{\unitlength}{1cm}
\begin{picture}(10,3.3)(1.375,-1)

\put(3.4,0.5){\Large{$H$}}
\put(9.4,0.5){\Large{$G[W]$}}

\put(2,1.5){\line(1,0){2}}
\put(1.91,1.7){$1$}
\put(2,1.5){\circle*{0.1}}
\put(2.91,1.7){$4$}
\put(3,1.5){\circle*{0.1}}
\put(3.91,1.7){$2$}
\put(4,1.5){\circle*{0.1}}
\put(3,1.5){\line(0,-1){1.5}}
\put(2,1.5){\line(2,-3){1}}
\put(2.91,-0.5){$3$}
\put(3,0){\circle*{0.1}}

\put(6,1.5){\line(1,0){4}}
\put(6,0){\line(1,0){1}}
\put(6,1.5){\line(0,-1){1.5}}
\put(7,1.5){\line(0,-1){1.5}}
\put(5.91,1.7){$6$}
\put(6,1.5){\circle*{0.1}}
\put(6.91,1.7){$3$}
\put(7,1.5){\circle*{0.1}}
\put(5.91,-0.5){$1$}
\put(6,0){\circle*{0.1}}
\put(6.91,-0.5){$8$}
\put(7,0){\circle*{0.1}}

\put(7.91,1.7){$2$}
\put(8,1.5){\circle*{0.1}}
\put(8.91,1.7){$7$}
\put(9,1.5){\circle*{0.1}}
\put(9.91,1.7){$4$}
\put(10,1.5){\circle*{0.1}}
\put(9,1.5){\line(0,-1){1.5}}
\put(8,1.5){\line(2,-3){1}}
\put(8.91,-0.5){$5$}
\put(9,0){\circle*{0.1}}

\put(7,-1){
\begin{tikzpicture} 
\draw[gray, thick,rounded corners=8pt]
(0,0) -- (2,0) -- (2.2,3.3) -- (-1.5,3.3) -- cycle;
\end{tikzpicture}
}

\end{picture}
\caption{A graph $H$ and an appearance of $H$.} \label{appfig}
\end{figure}

When $m$ is close to $3n$,
we will find appearances rather scarce
(since each involves a cut-edge),
and so we will instead often employ the concept of `triangulated appearances'
(a modified version of `$6$-appearances' from~\cite{dow}):

\begin{Definition} \label{6apps}
We say that a connected graph $H$ has a
\emph{triangulated appearance} at $W \subset V(G)$ if
(a) the increasing bijection from $V(H)$ to $W$
gives an isomorphism between $H$ and the induced subgraph $G[W]$ of $G$;
and (b) there are exactly six edges in $G$ between $W$ and the rest of $G$,
and these are of the form 
$E_{W} 
=\{ r_{1}v_{1}, v_{1}r_{2}, r_{2}v_{2}, v_{2}r_{3}, r_{3}v_{3}, v_{3}r_{1} \}$,
where $\{ r_{1}, r_{2}, r_{3} \} \subset W$ 
and $\{ v_{1}, v_{2}, v_{3} \} \subset V(G) \setminus W$,
and where $G[\{ v_{1}, v_{2}, v_{3} \}]$ is a triangle.
See Figure~\ref{6appfig}.

We shall call $E(G[W]) \cup E_{W}$ 
the \emph{total edge set} of the triangulated appearance,
and we shall say that two triangulated appearances are 
\emph{totally edge-disjoint}
if the two total edge sets do not share a common edge.
We say that two triangulated appearances at $W_{1}$ and $W_{2}$
are \emph{vertex-disjoint} if $W_{1} \cap W_{2} = \emptyset$.
Note that total edge-disjointness is consequently
a stronger condition than vertex-disjointness.

We say that a triangulated appearance is \emph{rooted}
if $r_{1}, r_{2}$ and $r_{3}$ are the three lowest labelled vertices in $W$. 
\end{Definition}

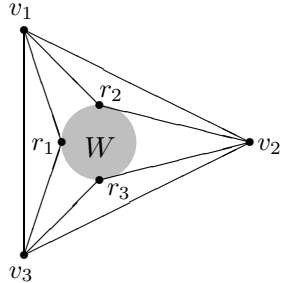
\begin{figure} [ht]
\setlength{\unitlength}{1cm}
\begin{picture}(10,3.6)(-3.5,-0.3)

\put(0.375,1){
\begin{tikzpicture}
\fill[lightgray] (0,0) circle (0.5);
\end{tikzpicture}
}

\put(0,0){\line(0,1){3}}
\put(0,0){\line(2,1){3}}
\put(0,3){\line(2,-1){3}}
\put(0,0){\line(1,3){0.5}}
\put(0,0){\line(1,1){1}}
\put(0,3){\line(1,-3){0.5}}
\put(0,3){\line(1,-1){1}}
\put(1,1){\line(4,1){2}}
\put(1,2){\line(4,-1){2}}
\put(0,0){\circle*{0.1}}
\put(0,3){\circle*{0.1}}
\put(3,1.5){\circle*{0.1}}
\put(0.5,1.5){\circle*{0.1}}
\put(1,1){\circle*{0.1}}
\put(1,2){\circle*{0.1}}
\put(0.8,1.3){\large{$W$}}
\put(-0.2,-0.3){$v_{3}$}
\put(-0.2,3.2){$v_{1}$}
\put(3.1,1.4){$v_{2}$}
\put(1.1,0.8){$r_{3}$}
\put(0.1,1.4){$r_{1}$}
\put(1,2.1){$r_{2}$}

\end{picture}
\caption{A triangulated appearance at $W$.} \label{6appfig}
\end{figure}

Another critical ingredient in many of our proofs
is the number of choices for where to insert an edge:

\begin{Definition}
Given any fixed $g \geq 0$
and any fixed graph $G$ with genus at most $g$,
we call a non-edge $e$ \emph{g-addable} in $G$
if the graph $G+e$ obtained by adding $e$ as an edge 
still has genus at most $g$,
and we let \emph{add$_{g}(G)$}
denote the set of g-addable non-edges of $G$
(note that the graph obtained by adding two edges in add$_{g}(G)$ 
may well have genus greater than $g$).
\end{Definition}

Observe that the set of $g$-addable edges always includes any edge 
between two vertices in different components
(see~\cite{add,bal,chap3,kangpan,msw2}
for many interesting results on the related topic of `bridge-addable' classes).

We now conclude this subsection 
by stating our definition for the `size' of a face
(which will be the topic of Section~\ref{facesize}):

\begin{Definition}
Given a particular embedding of a graph,
we shall use the \emph{size} of a face 
to mean the number of edges with a side in the face,
counting an edge twice if both sides are in the face.
\end{Definition}

\subsection{Key lemmas} \label{lemmas}

As mentioned,
we will now collect together various key lemmas
that will be used to prove our main results.

We start by stating two very useful ingredients
(Lemma~\ref{mainapps} and Lemma~\ref{main6apps})
concerning the number of appearances and triangulated appearances
(the proofs of which will be given in 
Sections~\ref{appsection} and~\ref{6appsection}):

\begin{Lemma} \label{mainapps}
Let $H$ be a (fixed) connected planar graph on $[|H|]$,
let $g \geq 0$ be a constant,
and let $m=m(n)$ satisfy 
$1 < \liminf \frac{m}{n} \leq \limsup \frac{m}{n} < 3$.
Then there exist $\alpha > 0$ and $N$ such that
\begin{equation*}
{\renewcommand{\arraystretch}{1.5}
\begin{array}{c}
\mathbb{P}
[S_{g}(n,m) 
\textrm{ will have at least $\alpha n$ vertex-disjoint appearances of $H$}] 
> 1 - e^{- \alpha n} 
\end{array}}
\end{equation*}
for all $n \geq N$.
\end{Lemma}

\begin{Lemma} \label{main6apps}
Let $T$ be a (fixed) planar triangulation,
let $g \geq 0$ be a constant,
and let $m=m(n)$ satisfy
$\liminf \frac{m}{n} > 1$.
Then there exist $\alpha > 0$ and $N$ such that 
\begin{equation*}
{\renewcommand{\arraystretch}{1.5}
\begin{array}{r}
\mathbb{P}
[S_{g}(n,m) 
\textrm
{ will have at least $\alpha n$ 
totally edge-disjoint triangulated appearances of $T$}] \\
> 1 - e^{- \alpha n} \textrm{ for all } n \geq N.
\end{array}}
\end{equation*}
\end{Lemma}

The following result,
recently given in~\cite{kmsfull} (see also~\cite{kms}),
will be of use to us as well
(note that the brace notation used here is to be understood as meaning that both conditions must be satisfied simultaneously):

\begin{Lemma} (\cite{kmsfull}, Theorem 5.2) \label{kmslemma}
Let $g \geq 0$ be a constant.
Then whp
the total number of vertices of $S_{g}(n,m)$ in non-multicyclic components 
and the total number of edges of $S_{g}(n,m)$ in non-multicyclic components 
are both
\begin{eqnarray*}
\Theta (n-m)
& \textrm{for} &
0 < n-m = 
\left\{ \begin{array}{ll}
\omega (n^{3/5}) \\
o (n),
\end{array} \right. \\
\Theta (n^{3/5})
& \textrm{for} &
|m-n| = O(n^{3/5}), \\
\textrm{and }
\Theta \left( \left( \frac{n}{m-n} \right) ^{3/2} \right)
& \textrm{for} &
0 < m-n =
\left\{ \begin{array}{ll}
\omega (n^{3/5}) \\
o \left( \frac{n}{(\log n)^{2/3}} \right).
\end{array} \right. \\
\end{eqnarray*} 
\end{Lemma}

We note that it is not known
whether the $o \left( \frac{n}{(\log n)^{2/3}} \right)$ condition in Lemma~\ref{kmslemma} can be amended to $o(n)$.

Information on the number of pendant edges
will be important in many of our proofs,
and we will often utilise the following result:

\begin{Lemma} \label{pendant}
Let $g \geq 0$ be a constant,
and let $m=m(n)$ satisfy
$0 < \liminf \frac{m}{n} \leq \limsup \frac{m}{n} < 3$.
Then there exist $\alpha > 0$ and $N$ such that
\begin{equation*}
{\renewcommand{\arraystretch}{1.5}
\begin{array}{c}
\mathbb{P}
[S_{g}(n,m) 
\textrm
{ will have at least $\alpha n$ pendant edges}] 
> 1 - e^{- \alpha n} 
\textrm{ for all } n \geq N.
\end{array}}
\end{equation*}
\end{Lemma}
\begin{sketch}
The case when 
$1 < \liminf \frac{m}{n} \leq \limsup \frac{m}{n} < 3$
already follows from Lemma~\ref{mainapps},
with $H$ as an isolated vertex.
For the remaining region when
$0 < \liminf \frac{m}{n} \leq \limsup \frac{m}{n} \leq 1$,
we may use the same proof as given for the planar case
in Lemma 15 of~\cite{dow},
as the details actually generalise to any genus $g \geq 0$.
\qed
\end{sketch}

We finish this section with three results on the number of addable edges.
Firstly,
recall that inserting an edge between any two vertices in different components
cannot increase the genus,
and so for $\limsup \frac{m}{n} < 1$ we obtain the following
(since there are at least $n-m$ components):

\begin{Lemma} \label{add1}
Let $g \geq 0$ be a constant,
and let $m=m(n)$ satisfy $\limsup \frac{m}{n} < 1$.
Then
\begin{displaymath}
|\emph{add}_{g}(S_{g}(n,m))| = \Theta (n^{2}). 
\end{displaymath} 
\end{Lemma}
\qed

Secondly,
note that we can also always insert an edge 
between a vertex in a tree/unicyclic component
and any non-adjacent vertex,
and so Lemma~\ref{kmslemma} provides us with another useful result:

\begin{Lemma} \label{add2}
Let $g \geq 0$ be a constant,
and let $m=m(n)$ satisfy $m \leq n + o \left( \frac{n}{(\log n)^{2/3}} \right)$.
Then whp
\begin{displaymath}
|\emph{add}_{g}(S_{g}(n,m))| = \omega (n).
\end{displaymath}
\end{Lemma}
\qed

For the planar case,
the number of addable edges is always at least $3n-6-m$,
since any planar graph can always be extended into a triangulation
without inserting any multi-edges.
This is not the case for graphs of higher genus,
but we may still obtain the following result 
as a corollary to Lemma~\ref{mainapps}
(e.g.~let $H=C_{4}$
and note that each appearance of $C_{4}$ provides us with two addable edges):

\begin{Lemma} \label{add3}
Let $g \geq 0$ be a constant,
and let $m=m(n)$ satsify 
$1 < \liminf \frac{m}{n} \leq \limsup \frac{m}{n} < 3$.
Then whp
\begin{displaymath}
|\emph{add}_{g}(S_{g}(n,m))| = \Omega (n).
\end{displaymath}
\end{Lemma}
\qed

For those who are interested,
a very detailed account of the number of addable edges
for the planar case is given in Section 5 of~\cite{dow}.

\section{Components: proof of Theorem~\ref{maincpt}} \label{cptsect}

We now come to the first main section of this paper,
where we look at the probability that $S_{g}(n,m)$
will have a component isomorphic to $H$,
for various fixed $H$.
The main feature of this section will be a proof of Theorem~\ref{maincpt},
but we will also collect together various other results on this topic.

We shall let $C_{H}(S_{g}(n,m))$
denote the number of components in $S_{g}(n,m)$ isomorphic to $H$,
and we shall let 
\begin{displaymath}
\mathbb{P}^{H}_{cpt} := \mathbb{P}[C_{H}(S_{g}(n,m)) \geq 1],
\end{displaymath}
i.e.~the probability that 
$S_{g}(n,m)$ will have at least one component isomorphic to $H$.
An informal summary of the results is given in Table~\ref{cpttab},
which shows that the asymptotic behaviour depends on both 
the type of component $H$ and on the ratio of $m$ to $n$.

\begin{table} [ht]
\centering
\caption{
A summary of~
$\mathbb{P}^{H}_{cpt} :=~
\mathbb{P}[S_{g}(n,m) \textrm{ will have a component isomorphic to } H]$.
} 
\label{cpttab}
{\renewcommand{\arraystretch}{1.5}
\begin{tabular}{|c||c|c|c|c|}
\hline
&
\textrm{\small{$e(H)<|H|$}} &
\textrm{\small{$e(H)=|H|$}} &
\multicolumn{2}{|c|}
{\textrm{\small{$e(H) > |H|$}}} \\
\cline{4-5}
&
&
&
\textrm{\small{$H$ planar}} &
\textrm{\small{$H$ non-planar}} \\
\hline
\hline
\textrm{\small{$\liminf \frac{m}{n} > 0$}} &
\small{$\mathbb{P}^{H}_{cpt} \to 1$} &
&
\small{$\mathbb{P}^{H}_{cpt} \to 0$}
&
\\
\textrm{\small{\& $m \leq n + o \left( \frac{n}{(\log n)^{2/3}} \right)$}} &
\textrm{\small{(Thm.~\ref{cpttree})}} &
\textrm{\small{$\liminf \mathbb{P}^{H}_{cpt} > 0$}} &
\textrm{\small{(Thm.~\ref{cptmulti})}} &
\\
\cline{1-2}
\cline{4-4}
\textrm{\small{$m \geq n + \Omega \left( \frac{n}{(\log n)^{2/3}} \right)$}} &
\small{unknown} &
\textrm{\small{$\limsup \mathbb{P}^{H}_{cpt} < 1$}} &
\small{unknown} &
\\
\textrm{\small{\& $m \leq n + o(n)$}} &
&
\textrm{\small{(Thm.~\ref{unicpt})}} &
&
\small{$\mathbb{P}^{H}_{cpt} \to 0$} \\
\cline{1-4}
\textrm{\small{$1 < \liminf \frac{m}{n}$}} &
\multicolumn{3}{|c|}
{\small{$\liminf \mathbb{P}^{H}_{cpt} > 0$ \textrm{}}} &
\textrm{\small{(Thm.~\ref{maincpt})}} \\
\textrm{\small{\& $\limsup \frac{m}{n} < 3$}} &
\multicolumn{3}{|c|}
{\small{$\limsup \mathbb{P}^{H}_{cpt} < 1$}} &
\\
&
\multicolumn{3}{|c|}
{\textrm{\small{(Thm.~\ref{cpt1to3})}}} &
\\
\cline{1-4}
\small{$\frac{m}{n} \to 3$} &
\multicolumn{3}{|c|}
{\small{$\mathbb{P}^{H}_{cpt} \to 0$ (Thm.~\ref{conn1})}} &
\\
\hline 
\end{tabular}} 
\end{table}

We start with a useful lemma on the number of isolated vertices:

\begin{Lemma} \label{fewisolated}
Let $g \geq 0$ be a constant,
and let $m=m(n)$ satisfy $\liminf \frac{m}{n} \geq 1$.
Then whp
$S_{g}(n,m)$ has $o(n)$ components, and hence $o(n)$ isolated vertices.
\end{Lemma}
\begin{sketch}
The lemma follows from the same proofs 
(using 
$|\mathcal{S}^{g}(n,m)| \geq |\mathcal{S}^{0}(n,m)|$
at relevant points)
as for Lemma 41 and Proposition 50 of~\cite{dow},
which give upper bounds of type $o(n)$ 
on the number of components in a random planar graph
when $m \geq n$ and $m=(1-o(1))n$, respectively.
\qed
\end{sketch}

We shall now employ Lemma~\ref{fewisolated} 
in the proof of the main result of this section
(Theorem~\ref{maincpt}):

\begin{cptproof}
The result for the cases when 
$\limsup \frac{m}{n} < 1$ and $\frac{m}{n} \to 3$
will be covered separately by Theorems~\ref{cptmulti} and~\ref{conn1},
so we will assume here that we have 
$1 \leq \liminf \frac{m}{n} \leq \limsup \frac{m}{n} < 3$.

Let $\epsilon > 0$.
We will aim to show that
$\mathbb{P}^{H}_{cpt} < \epsilon$
for all sufficiently large $n$.
Let $\alpha$ be as given by Lemma~\ref{pendant},
and let $\mathcal{G}_{n}$ denote the set of graphs in $\mathcal{S}^{g}(n,m)$
with a component isomorphic to $H$,
at least $\alpha n$ pendant edges,
and at most $\frac{\alpha \epsilon n}{8g}$ isolated vertices.
By Lemmas~\ref{pendant} and~\ref{fewisolated},
it will suffice to show
$\frac{|\mathcal{G}_{n}|}{|\mathcal{S}^{g}(n,m)|} < \frac{\epsilon}{2}$
for all sufficiently large $n$.

For each graph in $\mathcal{G}_{n}$,
let us insert an edge between a component isomorphic to $H$
and any vertex outside this component.
Note that we have $|H|(n-|H|)$ ways to do this,
and the overall graph will still have genus at most $g$.
Let us then also delete a pendant edge to create an isolated vertex
(or perhaps two).
Note that we still have at least $\alpha n - 2$ choices for this pendant edge.
Thus, 
we find that we can construct at least 
$|\mathcal{G}_{n}| |H|(n-|H|) (\alpha n - 2)$ 
(not necessarily distinct) graphs in $\mathcal{S}^{g}(n,m)$.

Let us now consider the amount of double-counting.
Given one of our constructed graphs,
there are at most $\frac{\alpha \epsilon n}{8g} + 2$ possibilities 
for the newly created isolated vertex,
and at most $n$ possibilities for the other endpoint of the deleted pendant edge.
There are then at most $|H|g$ possibilities
for the edge that was inserted,
since there can only be at most $|H|g$ pendant copies of $H$ in the graph
(using the observation that 
there can only be at most $g$ vertex-disjoint copies of $H$ in the graph,
since $H$ is non-planar,
together with the fact that any pendant copy of $H$ 
can only have a vertex in common with at most $|H|-1$ others,
by Observation~\ref{appinter}).
Thus,
we find that we have built
each graph at most $|H|g \left( \frac{\alpha \epsilon n}{8g} + 2 \right)n$ times.

Hence,
the number of distinct graphs (in $\mathcal{S}^{g}(n,m)$)
that we have constructed
must be at least
\begin{eqnarray*}
\frac{|\mathcal{G}_{n}| |H|(n-|H|) (\alpha n - 2)}
{|H|g \left( \frac{\alpha \epsilon n}{8g} + 2 \right)n}
& \geq &
\frac{|\mathcal{G}_{n}| |H| \frac{n}{2} \alpha n}
{|H|g \frac{\alpha \epsilon n}{4g}n}
\textrm{ for all sufficiently large $n$} \\
& = & |\mathcal{G}_{n}| \frac{2}{\epsilon},
\end{eqnarray*}
and so
$\frac{|\mathcal{G}_{n}|}{|\mathcal{S}^{g}(n,m)|}
< \frac{\epsilon}{2}$,
as desired.
\qed 
\end{cptproof}

In the remainder of this section,
we collect together other interesting results on
$\mathbb{P}^{H}_{cpt}$
for various different cases.
The proofs all follow those given in~\cite{dow} and~\cite{dowejc}
for a random planar graph,
and so we shall just state these theorems 
without providing full details of the proofs.
The only difference to the planar case
is that we only have the addability result of Lemma~\ref{add2}
for $m \leq n + o \left( \frac{n}{(\log n)^{2/3}} \right)$ here,
whereas the analogous result for planar graphs
is known to hold for $m \leq (1+o(1))n$,
and so Theorems~\ref{cpttree} and~\ref{cptmulti}
will consequently only be stated for $m \leq n + o \left( \frac{n}{(\log n)^{2/3}} \right)$ too.

We start with a stronger result than given in Table~\ref{cpttab},
showing that whp there are actually \emph{linearly} many components
isomorphic to any given tree if $\limsup \frac{m}{n} < 1$:

\begin{Theorem} \label{ntree}
Let $H$ be a (fixed) tree,
let $g \geq 0$ be a constant,
and let $m=m(n)$ satisfy 
$0 < \liminf \frac{m}{n} \leq \limsup \frac{m}{n} < 1$.
Then there exist $\alpha > 0$ and $N$ such that
\begin{displaymath}
\mathbb{P} 
[C_{H}(S_{g}(n,m)) \geq \alpha n] 
> 1 - e^{- \alpha n} 
\textrm{ for all } n \geq N.
\end{displaymath}
\end{Theorem}
\begin{sketch}
The (double-counting) proof involves deleting $|H|$ pendant edges,
building a component isomorphic to $H$ on the $|H|$ newly isolated vertices
(using $|H|-1$ edges),
and inserting an edge elsewhere in the graph
(see Theorem 11 of~\cite{dowejc} for the full proof of the analogous planar case).
We use Lemma~\ref{pendant} and Lemma~\ref{add1} to obtain the required numbers of pendant edges and $g$-addable edges.
\qed
\end{sketch}

By Lemma~\ref{fewisolated},
we certainly cannot expect to find linearly many components 
if $\frac{m}{n} \to 1$,
but 
it is still possible to show that whp there is at least one component
isomorphic to any given tree,
as long as $m \leq n + o \left( \frac{n}{(\log n)^{2/3}} \right)$:

\begin{Theorem} \label{cpttree}
Let $H$ be a (fixed) tree,
let $g \geq 0$ be a constant,
and let $m=m(n)$ satisfy $m \leq n + o \left( \frac{n}{(\log n)^{2/3}} \right)$.
Then 
\begin{displaymath}
\mathbb{P}^{H}_{cpt} 
\to 1 
\textrm{ as } n \to \infty.
\end{displaymath}
\end{Theorem}
\begin{sketch}
We adopt the same proof as with Theorem~\ref{ntree},
using Lemma~\ref{add2} instead of Lemma~\ref{add1}
(see Theorem 10 of~\cite{dowejc}
for details of the analogous planar case).
\qed
\end{sketch}

By contrast,
for the unicyclic case 
(i.e.~when $e(H)=|H|$)
we find that the probability is bounded away from $1$:

\begin{Theorem} \label{unicpt}
Let $H$ be a (fixed) connected unicyclic graph,
let $g \geq 0$ be a constant,
and let $m=m(n)$ satisfy
$0 < \liminf \frac{m}{n} \leq \limsup \frac{m}{n} < 3$.
Then 
\begin{eqnarray*}
\liminf \mathbb{P}^{H}_{cpt} & > & 0 \\
\textrm{and }
\limsup \mathbb{P}^{H}_{cpt} & < & 1.
\end{eqnarray*}
\end{Theorem}
\begin{sketch}
To show $\liminf \mathbb{P}^{H}_{cpt} > 0$,
we again apply the same proof as with Theorem~\ref{ntree},
but without inserting an extra edge elsewhere in the graph
(see Theorem 9 of~\cite{dowejc} 
for details of the analogous planar case).
To show $\limsup \mathbb{P}^{H}_{cpt} < 1$,
we use the $g=0$ proof of Theorem 13 of~\cite{dowejc},
which involves deleting an edge from a cycle in $H$
and then inserting an edge to join this component to the rest of the graph.
\qed
\end{sketch}

For the multicyclic case
(i.e.~when $e(H)>|H|$),
we find that the probability actually converges to $0$ for $m \leq n + o \left( \frac{n}{(\log n)^{2/3}} \right)$:

\begin{Theorem} \label{cptmulti}
Let $H$ be a (fixed) connected multicyclic graph,
let $g \geq 0$ be a constant,
and let $m=m(n)$ satisfy $m \leq n + o \left( \frac{n}{(\log n)^{2/3}} \right)$.
Then 
\begin{displaymath}
\mathbb{P}^{H}_{cpt}
\to 0
\textrm{ as } n \to \infty.
\end{displaymath}
\end{Theorem}
\begin{sketch}
The proof involves deleting two edges from a component isomorphic to $H$ without disconnecting it,
inserting an edge to join this component to the rest of the graph,
and also inserting an edge elsewhere
(see Theorem 12 of~\cite{dowejc}
for the full proof of the analogous planar case).
Here, we use Lemma~\ref{add2} to show that
the number of $g$-addable edges is sufficiently large to obtain our result by double-counting.
\qed
\end{sketch}

Moving into the region when $\liminf \frac{m}{n} > 1$,
the probability of containing a given component $H$
is bounded away from $0$ for all planar $H$
(cf.~the non-planar case of Theorem~\ref{maincpt}):

\begin{Theorem} \label{cpt1to3}
Let $H$ be a (fixed) connected planar graph,
let $g \geq 0$ be a constant,
and let $m=m(n)$ satisfy
$1 < \liminf \frac{m}{n} \leq \limsup \frac{m}{n} < 3$.
Then 
\begin{eqnarray*}
\liminf \mathbb{P}^{H}_{cpt} & > & 0 \\
\textrm{and }
\limsup \mathbb{P}^{H}_{cpt} & < & 1.
\end{eqnarray*}
\end{Theorem}
\begin{sketch}
For the lower bound,
we may delete the cut-edge from an appearance of $H$
and insert an edge elsewhere in the graph,
applying Lemma~\ref{mainapps} on the number of appearances
and Lemma~\ref{add3} on the number of $g$-addable edges
(see Theorem 8 of~\cite{dowejc} for details of a proof for the planar case).
The upper bound will follow from Theorem~\ref{conn0}.
\qed
\end{sketch}

The probability is also bounded away from $1$ for $\liminf \frac{m}{n} > 1$,
due to the following result on the connectivity of $S_{g}(n,m)$:

\begin{Theorem} \label{conn0}
Let $g \geq 0$ be a constant,
and let $m=m(n)$ satisfy $\liminf \frac{m}{n} > 1$.
Then
\begin{displaymath}
\liminf \mathbb{P}
[S_{g}(n,m) \textrm{ will be connected}] > 0.
\end{displaymath}
\end{Theorem}
\begin{sketch}
The case when $\frac{m}{n} \to 3$ will follow from Theorem~\ref{conn1}.
For the case when $\limsup \frac{m}{n} < 3$,
we may use the proof of Lemma 42 of~\cite{dow},
which involves deleting a non-cut-edge from an appearance of a given planar graph (applying Lemma~\ref{mainapps}),
and inserting an edge between two components
(using Lemma~\ref{fewisolated} to bound the amount of double-counting).
\qed
\end{sketch}

Finally, if $\frac{m}{n} \to 3$,
we actually find that $S_{g}(n,m)$ is connected whp:

\begin{Theorem} \label{conn1}
Let $g \geq 0$ be a constant,
and let $m=m(n)$ satisfy $\frac{m}{n} \to 3$ as $n \to \infty$.
Then
\begin{displaymath}
\mathbb{P}
[S_{g}(n,m) \textrm{ will be connected}] \to 1
\textrm{ as } n \to \infty.
\end{displaymath}
\end{Theorem}
\begin{proof}
We may employ the method of proof of the $g=0$ case from Theorem 14 of~\cite{dowejc}
(with some small differences).
This involves first establishing that (when $\frac{m}{n} \to 3$)
every graph in $\mathcal{S}^{g}(n,m)$ must have
(i)~$\Omega (n)$ triangles containing a vertex with degree at most $6$
and (ii)~$o(n)$ cut-edges.

Hence, let $G \in \mathcal{S}^{g}(n,m)$, and let us start by considering how many triangles in $G$ contain at least one vertex with degree at most $6$.

First, note that (assuming $n \geq 3$)
$G$ may be extended to a triangulation of genus $g$ by inserting $3n-6+6g-m = o(n)$ `phantom' edges
(observe that such a triangulation may be a multi-graph, rather than a simple graph).
Let $d_{i}$ denote the number of vertices of degree $i$ in such a triangulation.
Then 
\begin{eqnarray*}
7 \sum_{i \geq 7} d_{i} & \leq & \sum_{i \geq 1} i d_{i} \\
& = & 2(3n-6+6g).
\end{eqnarray*}
Thus, $\sum_{i \geq 7} d_{i} \leq \frac{6n-12+12g}{7}$, and so $\sum_{i \leq 6} d_{i} \geq \frac{n+12-12g}{7}$.

Let us call a triangle `good' if it contains a vertex with degree at most $6$.
Since each (triangular) face contains at most three such vertices,
we find that our triangulation must have at least 
$\frac{n+12-12g}{21}$
faces that are good triangles.
Note that each of our $o(n)$ phantom edges is in at most two faces of the triangulation,
and so our original graph $G$ must also contain at least $\frac{n}{21} + o(n)$ of these good triangles
(note that these triangles will still be `good' in $G$,
since the degrees of the vertices will be at most what they were in the triangulation).

We will now consider how many cut-edges a graph in $\mathcal{S}^{g}(n,m)$ may have.
If we delete all $c$ cut-edges,
then the remaining graph will consist of $b$, say, components,
each of which is either $2$-edge-connected or is an isolated vertex.
Note that the graph formed by condensing each of these components to a single node and re-inserting the cut-edges must be acyclic,
so $c \leq b-1$.
Label these components $1,2, \ldots b$, let $n_{i}$ denote the number of vertices in component $i$,
and let $g_{i}$ denote the genus of component $i$
(note that the overall genus is equal to the sum of the genera of these components).
Observe that the number of edges in component $i$ is at most $3n_{i}-6+6g_{i}$ if $n_{i} \geq 3$ and is $0=3n_{i}-3$ otherwise
(since $n_{i} < 3$ implies that $n_{i} = 1$).
Thus, 
since $\max \{ 3n_{i} - 6 + 6g_{i}, 3 n_{i} - 3 \}
\leq 3n_{i} - 3 + 6g_{i}$,
we have
\begin{eqnarray*}
m & \leq & \sum_{i=1}^{b} (3n_{i}-3+6g_{i}) + c \\
& \leq & 3n-3b+6g+c \\
& < & 3n+6g-2c,
\end{eqnarray*}
and so $c < \frac{3n-m+6g}{2} = o(n)$.

We now come to the main part of the proof.
Let $\mathcal{G}_{n}$ denote the set of graphs in $\mathcal{S}^{g}(n,m)$ that are not connected,
and choose a graph $G \in \mathcal{G}_{n}$.
Choose a good triangle $uvw$ in $G$ 
with $\deg(u) \leq 6$
(at least $\frac{n}{21} + o(n)$ choices),
and delete the edge $vw$.
Then insert an edge between two vertices in different components
--- we have $a$, say, choices for this edge.

Note that
the number of possible edges between disjoint sets $X$ and $Y$ is $|X||Y|$, and
if $|X| \leq |Y|$ then $|X||Y| > (|X|-1)(|Y|+1)$.
Hence, it follows that the number of choices for the edge to insert is minimised when
we have one isolated vertex and
one component of $n-1$ vertices,
and so $a \geq n-1$.
Thus,
we find that we can construct at least
$|\mathcal{G}_{n}| \left( \frac{n}{21} + o(n) \right) (n-1)$
(not necessarily distinct) graphs in $\mathcal{S}^{g}(n,m)$.

Let us now consider the amount of double-counting.
Recall that we have shown that \emph{every} graph in $\mathcal{S}^{g}(n,m)$ has $o(n)$ cut-edges.
Hence, given one of our constructed graphs,
there are at most $o(n)$ possibilities for which edge was inserted,
since it must be a cut-edge.
There are then at most $\left( ^{6}_{2} \right)n$ possibilities for where the deleted edge was originally,
since it must have been between two neighbours of a vertex with degree at most $6$
(we have at most $n$ possibilities for this vertex, and then
at most $\left( ^{6}_{2} \right)$ possibilities for its neighbours).
Thus, we find that we have built each graph at most $o(n^{2})$ times.

Hence,
the number of distinct graphs (in $\mathcal{S}^{g}(n,m)$)
that we have constructed must be at least
$\frac{\frac{1}{21}n^{2} + o \left( n^{2} \right)}{ o \left( n^{2} \right)}
|\mathcal{G}_{n}|$,
and so
$\frac{|\mathcal{G}_{n}|}{|\mathcal{S}^{g}(n,m)|} \leq \frac{o \left( n^{2} \right)}{\Theta \left( n^{2} \right)}
\to 0$ as $n \to \infty$.
\end{proof}

\section{Subgraphs: proof of Theorem~\ref{mainsub}} \label{subsect}

In this section,
we look at the probability that $S_{g}(n,m)$ will have a copy of $H$
(i.e.~a \emph{subgraph} isomorphic to $H$),
for various fixed $H$.
We shall let $S_{H}(S_{g}(n,m))$ denote
the size of the largest set of vertex-disjoint copies of $H$ in $S_{g}(n,m)$,
and we shall let
\begin{displaymath}
\mathbb{P}^{H}_{sub} :=
\mathbb{P}[S_{H}(S_{g}(n,m)) \geq 1],
\end{displaymath}
i.e.~the probability 
that $S_{g}(n,m)$ will have at least one copy of $H$.

The main difference to the results on components is that
we find there is a copy of any given planar graph 
whp when $\liminf \frac{m}{n} > 1$,
as stated in Theorem~\ref{mainsub}.
In Table~\ref{subtab},
we again provide a summary.

\begin{table} [ht]
\centering
\caption{
A summary of 
$\mathbb{P}^{H}_{sub} :=
\mathbb{P}[S_{g}(n,m) \textrm{ will have a copy of } H]$.
} 
\label{subtab}
{\renewcommand{\arraystretch}{1.5}
\begin{tabular}{|c||c|c|c|c|}
\hline
&
\textrm{\small{$e(H)<|H|$}} &
\textrm{\small{$e(H)=|H|$}} &
\multicolumn{2}{|c|}{
\textrm{\small{$e(H) > |H|$}}
} \\
\cline{4-5}
&
&
&
\textrm{\small{$H$ planar}} &
\textrm{\small{$H$ non-planar}} \\
\hline
\hline
\textrm{\small{$0 < \liminf \frac{m}{n}$}} &
&
\textrm{\small{$\liminf \mathbb{P}^{H}_{sub} > 0$}} &
\multicolumn{2}{|c|}
{\small{$\mathbb{P}^{H}_{sub} \to 0$}} 
\\
\textrm{\small{\& $\limsup \frac{m}{n} < 1$}} &
&
\textrm{\small{$\liminf \mathbb{P}^{H}_{sub} < 1$}} &
\multicolumn{2}{|c|}
{\textrm{\small{(Thm.~\ref{multisub})}}} 
\\
&
\small{$\mathbb{P}^{H}_{sub} \to 1$}  &
\small{(Thms.~\ref{unicpt} \&~\ref{unisubto1})} &
\multicolumn{2}{|c|}{}
\\
\cline{1-1}
\cline{3-5}
\textrm{\small{$m \geq n - o(n)$ \&}} &
\small{(Thm.~\ref{cpttree})} &
&
&
\\
\textrm{\small{ $m \leq n + o \left( \frac{n}{(\log n)^{2/3}} \right)$}} &
&
\small{$\mathbb{P}^{H}_{sub} \to 1$} &
\small{unknown} &
\\
\cline{1-2}
\textrm{\small{$m \geq n + \Omega \left( \frac{n}{(\log n)^{2/3}} \right)$}} &
\textrm{\small{unknown}} &
\small{(Thm.~\ref{unisubat1})} &
&
\small{unknown} \\
\textrm{\small{\& $m \leq n + o(n)$}} &
&
&
&
\\
\cline{1-4}
\small{$\liminf \frac{m}{n} > 1$} &
\multicolumn{3}{|c|}
{\small{$\mathbb{P}^{H}_{sub} \to 1$ (Thm.~\ref{mainsub}) }} &
\\
\hline
\end{tabular}}
\end{table}

We start with a stronger result than specified in Theorem~\ref{mainsub},
namely that whp there are actually linearly many vertex-disjoint copies
of any given planar graph $H$ if $\liminf \frac{m}{n} > 1$
(note that this follows immediately from Lemma~\ref{main6apps},
by taking $T$ to be any planar triangulation containing a copy of $H$):

\begin{Theorem} \label{sub1to3}
Let $H$ be a (fixed) connected planar graph,
let $g \geq 0$ be a constant,
and let $m=m(n)$ satisfy 
$\liminf \frac{m}{n} > 1$.
Then there exist $\alpha > 0$ and $N$ such that
\begin{displaymath}
\mathbb{P}[S_{H}(S_{g}(n,m)) \geq \alpha n] 
> 1 - e^{- \alpha n} 
\textrm{ for all } n \geq N. 
\end{displaymath} 
\end{Theorem} 
\qed

The other results in Table~\ref{subtab} 
follow either directly from Table~\ref{cpttab}
(note that the existence of a component isomorphic to $H$ 
implies the existence of a copy of $H$)
or from the same proofs as for the random planar graph 
(see~\cite{dow,dowejc}),
and so we shall again just state these theorems 
without providing full details of the proofs.

If $H$ is multicyclic
(i.e.~$e(H)>|H|$),
we find that whp we have no copies of $H$ when $\limsup \frac{m}{n} < 1$:

\begin{Theorem} \label{multisub}
Let $H$ be a (fixed) connected multicyclic graph,
let $g \geq 0$ be a constant,
and let $m=m(n)$ satisfy $\limsup \frac{m}{n} < 1$.
Then
\begin{displaymath}
\mathbb{P}^{H}_{sub} \to 0
\textrm{ as } n \to \infty.
\end{displaymath}
\end{Theorem}
\begin{sketch}
We may apply the proof of Theorem 22 of~\cite{dowejc}.
This involves deleting all $e(H)$ edges from a copy of $H$
and inserting them elsewhere in the graph
(using Lemma~\ref{add1} on the number of $g$-addable edges).
The double-counting is then limited by the fact that the original site of the copy had only $|H| < e(H)$ vertices.
\qed
\end{sketch}

If $H$ is unicyclic
(i.e.~$e(H)=|H|$),
we already know (from Theorem~\ref{unicpt})
that such a result would not be true.
However,
we do find that the probability is bounded away from $1$ 
when $\limsup \frac{m}{n} < 1$:

\begin{Theorem} \label{unisubto1}
Let $H$ be a (fixed) connected unicyclic graph,
let $g \geq 0$ be a constant,
and let $m=m(n)$ satisfy $\limsup \frac{m}{n} < 1$.
Then
\begin{displaymath}
\limsup \mathbb{P}^{H}_{sub} < 1.
\end{displaymath}
\end{Theorem}
\begin{sketch}
We follow the proof of Theorem 18 of~\cite{dowejc}.
The first part of this proof involves showing that the probability of having many copies of $H$ is small, by transferring the edges of such a copy to some isolated vertices to construct a component isomorphic to $H$
(using Theorem~\ref{ntree} with $|H|=1$
to show that there are many isolated vertices,
and our upper bound on $\mathbb{P}^{H}_{cpt}$ for unicyclic $H$ from Theorem~\ref{unicpt} to bound the amount of double-counting).
The remainder of the proof then involves destroying copies of $H$ one-by-one by deleting the edges from them and inserting edges between components.
\qed
\end{sketch}

By contrast,
we find that the probability actually converges to $1$ 
when $\frac{m}{n} \to~1$:

\begin{Theorem} \label{unisubat1}
Let $H$ be a (fixed) connected unicyclic graph,
let $g \geq 0$ be a constant,
and let $m=m(n)$ satisfy $\frac{m}{n} \to 1$.
Then
\begin{displaymath}
\mathbb{P}^{H}_{sub} \to 1
\textrm{ as } n \to \infty.
\end{displaymath}
\end{Theorem}
\begin{sketch}
The proof involves deleting $|H|$ pendant edges
(we have many choices for these, by Lemma~\ref{pendant}),
and then using $|H|-1$ of the newly isolated vertices
to convert another pendant edge into an appearance of $H$
(leaving one extra isolated vertex, and
applying Lemma~\ref{fewisolated} to bound the amount of double-counting).
See Theorem 21 of~\cite{dowejc} for full details of the analogous planar proof.
\qed
\end{sketch}

\section{Maximum degree: proof of Theorem~\ref{mainmaxdeg}} \label{maxdegsect}

In this section,
we investigate the maximum degree of $S_{g}(n,m)$
(recall that this is denoted by $\Delta (S_{g}(n,m))$),
providing the results stated earlier in Theorem~\ref{mainmaxdeg}.
For the case when 
$0 < \liminf \frac{m}{n} \leq \limsup \frac{m}{n} < \frac{1}{2}$,
we may simply use the fact that whp $\Delta (S_{g}(n,m)) = \Delta (G(n,m))$,
and the latter is already known to be 
$\Theta \left( \frac{\ln n}{\ln \ln n} \right)$ whp
(see, for example, Theorem 3.7 of~\cite{bol}).
Hence, we shall concentrate here on the region 
$\liminf \frac{m}{n} \geq \frac{1}{2}$.

For the random graph $S_{g}(n)$,
it is shown in~\cite{mcdr}
that the maximum degree is $\Theta (\ln n)$ whp,
and we find that we obtain the same result for our random graph $S_{g}(n,m)$
when $\liminf \frac{m}{n} > 1$.
The issue of finding tight bounds for the region 
$\frac{1}{2} \leq \liminf \frac{m}{n} \leq \limsup \frac{m}{n} \leq 1$
is left as an open problem.

This section will consist of two main results:
in Theorem~\ref{maxdeg},
we prove an upper bound of $O(\ln n)$ for all $m$;
and then,
in Theorem~\ref{maxdeg2},
we prove a lower bound of $\Omega (\ln n)$
for the case when $\liminf \frac{m}{n} > 1$.
A summary is given in Table~\ref{maxdegtab}.

\begin{table} [ht]
\centering
\caption{
A summary of the maximum degree of $S_{g}(n,m)$.
}
\label{maxdegtab}
{\renewcommand{\arraystretch}{1.5}
\begin{tabular}{|c||c|}
\hline
Range of $m$ &
$\Delta (S_{g}(n,m))$ \\
\hline
\hline
$0 < \liminf \frac{m}{n}$
&
$\Theta \left(\frac{\ln n}{\ln \ln n}\right)$ \textrm{whp}
\\
$\& \limsup \frac{m}{n} < \frac{1}{2}$
&
(from $G(n,m)$) \\
\hline
$\frac{1}{2} \leq \liminf \frac{m}{n}$ 
&
$O (\ln n)$ \textrm{whp} \\
$\& \limsup \frac{m}{n} \leq 1$
&
(Thm.~\ref{maxdeg}) \\
\hline
$\liminf \frac{m}{n} > 1$
&
$\Theta (\ln n)$ \textrm{whp} \\
&
(Thms.~\ref{maxdeg} \&~\ref{maxdeg2}) \\
\hline
\end{tabular}}
\end{table}

We start by stating a useful result on pendant vertices:

\begin{Lemma} \label{pendant2}
Let $g \geq 0$ be a constant,
and let $m=m(n)$ be any function.
Then whp 
each vertex of $S_{g}(n,m)$ is adjacent to
at most $\frac{2 \ln n}{\ln \ln n}$ pendant vertices.
\end{Lemma}
\begin{sketch}
The proof is identical to that of the analogous result for $S_{g}(n)$
given in Lemma 2.2 of~\cite{mcdr}.
\qed
\end{sketch}

We may now proceed to our aforementioned upper bound
for the maximum degree:

\begin{Theorem} \label{maxdeg}
Let $g \geq 0$ be a constant,
and let $m=m(n)$ be any function.
Then whp
\begin{displaymath}
\Delta (S_{g}(n,m)) = O(\ln n).
\end{displaymath}
\end{Theorem}
\begin{proof}
(i) Case when $\limsup \frac{m}{n} < \frac{1}{2}$:

This follows from standard results on $G(n,m)$. \\

(ii) Case when 
$\frac{1}{2} \leq \liminf \frac{m}{n} \leq \limsup \frac{m}{n} < 3$: 

Let $\alpha > 0$ be as given by Lemma~\ref{pendant},
let $C > \frac{6}{\alpha \ln 3}$,
and let $\mathcal{G}_{n}$
denote the set of graphs in $\mathcal{S}^{g}(n,m)$
where each vertex is adjacent 
to at most $\frac{2 \ln n}{\ln \ln n}$ pendant edges,
where there are at least $\alpha n$ pendant edges in total,
and where the maximum degree is at least $C \ln n$.
By Lemmas~\ref{pendant} and~\ref{pendant2},
it will suffice to show
$\frac{|\mathcal{G}_{n}|}{|\mathcal{S}^{g}(n,m)|} \to 0$ as $n \to \infty$.

Before we continue,
for each graph in $\mathcal{G}_{n}$
let us fix a particular embedding.

Now let us take one of these graphs,
and let $v$ be a vertex with $\deg (v) = d \geq C \ln n$.
Given our particular embedding of the graph,
let us denote the neighbours of $v$ in clockwise order
(in terms of how the edges leave $v$)
as $v_{1}, v_{2}, \ldots, v_{d}$,
where $v_{d}$ is the vertex with largest label.

Now let $a = \lfloor \frac{\alpha C \ln n}{6} \rfloor$,
and let us choose $a$ of the neighbours of $v$
(at least $\left( ^{d} _{a} \right) 
\geq \left( \frac{d}{a} \right)^{a}
\geq \left( \frac{6}{\alpha} \right)^{a}$ choices)
and $a$ \emph{ordered} pendant edges not adjacent to $v$
(at least $\left( \alpha n - \frac{2 \ln n}{\ln \ln n} \right)^{a}
\geq \left( \frac{\alpha n}{2} \right)^{a}$ choices for large $n$).
Let us denote the chosen neighbours of $v$ 
in our clockwise order as $v_{i_{1}}, v_{i_{2}}, \ldots, v_{i_{a}}$,
where $v_{i_{1}}$ is such that 
$v_{d} \in \{ v_{i_{1}}, v_{i_{1}+1}, \ldots, v_{i_{2}-1} \}$,
and let us denote the ordered chosen pendant vertices as
$u_{1}, u_{2}, \ldots, u_{a}$.
Delete all $d$ edges incident to $v$,
and also delete the $a$ chosen pendant edges.

Now consider the $a$ vertex sets
$\{ v_{i_{1}}, v_{i_{1}+1}, \ldots, v_{i_{2}-1} \},
\{ v_{i_{2}}, v_{i_{2}+1}, \ldots, v_{i_{3}-1} \}$,
$\ldots$,
$\{ v_{i_{a}}, v_{i_{a}+1}, \ldots, v_{i_{1}-1} \}$.
For all $j$,
let us join $u_{j}$ to all vertices in the $j$th set,
and let us then join each $u_{j}$ to $v$
(observe that we now have $m$ edges again in total,
and the genus cannot have increased).
Thus,
we find that we can construct at least
$|\mathcal{G}_{n}|(3n)^{a}$
(not necessarily distinct) graphs in $\mathcal{S}^{g}(n,m)$.
See Figure~\ref{maxdegfig1}.

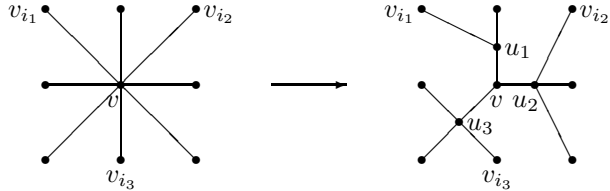
\begin{figure} [ht]
\setlength{\unitlength}{1cm}
\begin{picture}(10,2.3)(-1.5,-0.3)

\put(0,0){\line(1,1){2}}
\put(0,1){\line(1,0){2}}
\put(0,2){\line(1,-1){2}}
\put(1,0){\line(0,1){2}}
\put(0,0){\circle*{0.1}}
\put(0,1){\circle*{0.1}}
\put(0,2){\circle*{0.1}}
\put(1,0){\circle*{0.1}}
\put(1,1){\circle*{0.1}}
\put(1,2){\circle*{0.1}}
\put(2,0){\circle*{0.1}}
\put(2,1){\circle*{0.1}}
\put(2,2){\circle*{0.1}}
\put(-0.5,1.9){$v_{i_{1}}$}
\put(2.1,1.9){$v_{i_{2}}$}
\put(0.8,-0.3){$v_{i_{3}}$}
\put(0.8,0.7){$v$}

\put(3,1){\vector(1,0){1}}

\put(5,0){\line(1,1){1}}
\put(6,1){\line(1,0){1}}
\put(5,2){\line(2,-1){1}}
\put(6,1){\line(0,1){1}}
\put(5,0){\circle*{0.1}}
\put(5,1){\circle*{0.1}}
\put(5,2){\circle*{0.1}}
\put(6,0){\circle*{0.1}}
\put(6,1){\circle*{0.1}}
\put(6,2){\circle*{0.1}}
\put(7,0){\circle*{0.1}}
\put(7,1){\circle*{0.1}}
\put(7,2){\circle*{0.1}}
\put(4.5,1.9){$v_{i_{1}}$}
\put(7.1,1.9){$v_{i_{2}}$}
\put(5.8,-0.3){$v_{i_{3}}$}
\put(5.9,0.7){$v$}

\put(6,1.5){\circle*{0.1}}
\put(6.5,1){\circle*{0.1}}
\put(5.5,0.5){\circle*{0.1}}
\put(5,1){\line(1,-1){1}}
\put(6.5,1){\line(1,2){0.5}}
\put(6.5,1){\line(1,-2){0.5}}
\put(6.1,1.4){$u_{1}$}
\put(6.2,0.7){$u_{2}$}
\put(5.6,0.4){$u_{3}$}

\end{picture}
\caption{
Using former pendant vertices $u_{1}$, $u_{2}$ and $u_{3}$.} \label{maxdegfig1}
\end{figure}

Now let us consider the amount of double-counting.
We need to first identify $v$
(at most $n$ possibilities),
after which we can then determine 
the \emph{unordered} sets
$\{ u_{1}, u_{2}, \ldots, u_{a} \}$
and $\{ v_{1}, v_{2}, \ldots, v_{d} \}$
as being the neighbours and `distance 2 neighbours' of $v$.
We then also need to determine the original neighbours
of $u_{1}, u_{2}, \ldots, u_{a}$
(at most $n^{a}$ possibilities),
after which we then know the original graph
and hence the original embedding.
From this,
we can then determine the order of $v_{1}, v_{2}, \ldots, v_{d}$,
and hence also the order of $u_{1}, u_{2}, \ldots, u_{a}$.
Thus,
we find that we have built each graph at most $n^{a+1}$ times.

Hence,
the number of distinct graphs (in $\mathcal{S}^{g}(n,m)$)
that we have constructed must be at least
$|\mathcal{G}_{n}| \frac{(3n)^{a}}{n^{a+1}}
= |\mathcal{G}_{n}| \frac{3^{a}}{n}$,
and so 
\begin{eqnarray*}
\frac{|\mathcal{G}_{n}|}{|\mathcal{S}^{g}(n,m)|} 
& \leq & \frac{n}{3^{a}} \\
& = & n^{1- \frac{\alpha C \ln 3}{6} + o(1)} \\
& \to & 0 \textrm{ as } n \to \infty
\textrm{ since } C > \frac{6}{\alpha \ln 3}. \\
\end{eqnarray*}

(iii) Case when $\frac{m}{n} \to 3$: 

By Lemma~\ref{main6apps},
we know $S_{g}(n,m)$
has a set of
at least $\alpha n$ totally edge-disjoint triangulated appearances of $K_{4}$
with high probability.
The proof is then similar to the previous case,
but we choose 
$a$ ordered totally edge-disjoint triangulated appearances of $K_{4}$
instead of pendant edges,
and we use the degree $3$ vertex 
from each one of these chosen triangulated appearances
as our $u_{1}, u_{2}, \ldots, u_{a}$
(deleting all $3a$ edges incident to these).
Note that none of these degree $3$ vertices
can have been adjacent to $v$ as long as $d > 5$.
In order to maintain the correct number of edges overall,
we also insert additional edges 
$u_{1}u_{2}, u_{2}u_{3}, \ldots, u_{a}u_{1}$
and $u_{1}v_{i_{2}}, u_{2}v_{i_{3}}, \ldots, u_{a}v_{i_{1}}$.
See Figures~\ref{ujfig} and~\ref{maxdegfig2}.

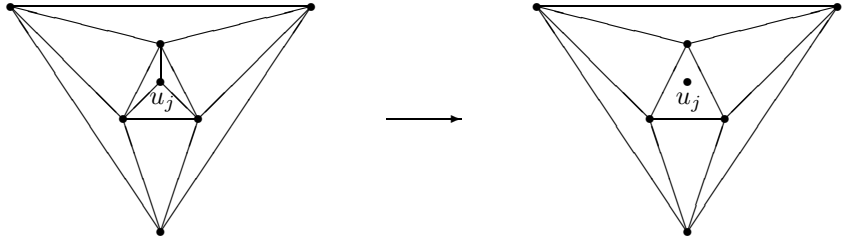
\begin{figure} [ht]
\setlength{\unitlength}{1cm}
\begin{picture}(10,3)(0.5,0)
\put(1.5,1.5){\line(1,0){1}}
\put(1.5,1.5){\line(1,1){0.5}}
\put(1.5,1.5){\line(1,2){0.5}}
\put(2,2){\line(0,1){0.5}}
\put(2.5,1.5){\line(-1,1){0.5}}
\put(2.5,1.5){\line(-1,2){0.5}}
\put(1.5,1.5){\circle*{0.1}}
\put(2.5,1.5){\circle*{0.1}}
\put(2,2){\circle*{0.1}}
\put(2,2.5){\circle*{0.1}}

\put(0,3){\line(1,0){4}}
\put(2,0){\line(-2,3){2}}
\put(2,0){\line(2,3){2}}
\put(2,0){\circle*{0.1}}
\put(0,3){\circle*{0.1}}
\put(4,3){\circle*{0.1}}

\put(2,2.5){\line(-4,1){2}}
\put(2,2.5){\line(4,1){2}}
\put(2,0){\line(-1,3){0.5}}
\put(2,0){\line(1,3){0.5}}
\put(0,3){\line(1,-1){1.5}}
\put(4,3){\line(-1,-1){1.5}}

\put(1.85,1.7){$u_{j}$}

\put(5,1.5){\vector(1,0){1}}

\put(8.5,1.5){\line(1,0){1}}
\put(8.5,1.5){\line(1,2){0.5}}
\put(9.5,1.5){\line(-1,2){0.5}}
\put(8.5,1.5){\circle*{0.1}}
\put(9.5,1.5){\circle*{0.1}}
\put(9,2){\circle*{0.1}}
\put(9,2.5){\circle*{0.1}}

\put(7,3){\line(1,0){4}}
\put(9,0){\line(-2,3){2}}
\put(9,0){\line(2,3){2}}
\put(9,0){\circle*{0.1}}
\put(7,3){\circle*{0.1}}
\put(11,3){\circle*{0.1}}

\put(9,2.5){\line(-4,1){2}}
\put(9,2.5){\line(4,1){2}}
\put(9,0){\line(-1,3){0.5}}
\put(9,0){\line(1,3){0.5}}
\put(7,3){\line(1,-1){1.5}}
\put(11,3){\line(-1,-1){1.5}}

\put(8.85,1.7){$u_{j}$}
\end{picture}
\caption{
Using a triangulated appearance of $K_{4}$ 
to find an appropriate vertex $u_{j}$.
} \label{ujfig}
\end{figure}

\begin{figure} [ht]
\setlength{\unitlength}{1cm}
\begin{picture}(10,2.3)(-1.375,-0.3)

\put(0,0){\line(1,1){2}}
\put(0,1){\line(1,0){2}}
\put(0,2){\line(1,-1){2}}
\put(1,0){\line(0,1){2}}
\put(0,0){\circle*{0.1}}
\put(0,1){\circle*{0.1}}
\put(0,2){\circle*{0.1}}
\put(1,0){\circle*{0.1}}
\put(1,1){\circle*{0.1}}
\put(1,2){\circle*{0.1}}
\put(2,0){\circle*{0.1}}
\put(2,1){\circle*{0.1}}
\put(2,2){\circle*{0.1}}
\put(-0.5,1.9){$v_{i_{1}}$}
\put(2.1,1.9){$v_{i_{2}}$}
\put(0.8,-0.3){$v_{i_{3}}$}
\put(0.8,0.7){$v$}

\put(3,1){\vector(1,0){1}}

\put(5,0){\line(1,1){1}}
\put(6,1){\line(1,0){1.25}}
\put(5,2){\line(2,-1){1}}
\put(6,1){\line(0,1){1}}
\put(5,0){\circle*{0.1}}
\put(5,1){\circle*{0.1}}
\put(5,2){\circle*{0.1}}
\put(6,0){\circle*{0.1}}
\put(6,1){\circle*{0.1}}
\put(6,2){\circle*{0.1}}
\put(7,0){\circle*{0.1}}
\put(7.25,1){\circle*{0.1}}
\put(7,2){\circle*{0.1}}
\put(4.5,1.9){$v_{i_{1}}$}
\put(7.1,1.9){$v_{i_{2}}$}
\put(5.8,-0.3){$v_{i_{3}}$}
\put(5.9,0.775){$v$}

\put(6,1.5){\circle*{0.1}}
\put(6.5,1){\circle*{0.1}}
\put(5.5,0.5){\circle*{0.1}}
\put(5,1){\line(1,-1){1}}
\put(6.5,1){\line(1,2){0.5}}
\put(6.5,1){\line(1,-2){0.5}}
\put(6.1,1.4){$u_{1}$}
\put(6.6,0.775){$u_{2}$}
\put(5.6,0.4){$u_{3}$}

\put(5.5,0.5){\line(-1,3){0.5}}
\put(5.5,0.5){\line(1,2){0.5}}
\put(6,1.5){\line(1,-1){0.5}}
\put(6,1.5){\line(2,1){1}}
\put(6.5,1){\line(-2,-1){1}}
\put(6.5,1){\line(-1,-2){0.5}}

\end{picture}
\caption{Using former degree $3$ vertices $u_{1}$, $u_{2}$ and $u_{3}$.
} \label{maxdegfig2}
\end{figure}
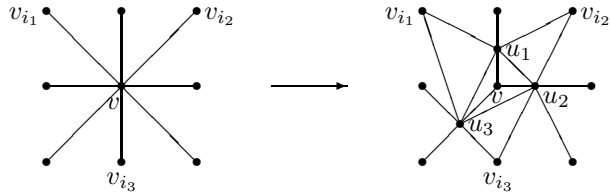

The calculations are then similar to before.
The only difference is that
when determining the original neighbours 
of $u_{2}, u_{3}, \ldots, u_{a}$,
we now have at most $\left( \frac{7}{3} n \right)^{a}$ possibilities
instead of $n^{a}$,
since we need to look for what will be triangulated appearances of $K_{3}$
(note there are at most 
$\frac{n}{3}$ vertex-disjoint triangulated appearances of $K_{3}$,
and each can have a vertex in common with at most $6$ others,
since there are only $6$ other triangles 
touching any triangulated appearance).
Hence,
we obtain
\begin{eqnarray*}
\phantom{wwwwwwwww}
\frac{|\mathcal{G}_{n}|}{|\mathcal{S}^{g}(n,m)|} 
& \leq & \frac{\left(\frac{7}{3}\right)^{a}n}{3^{a}} \\
& = & \frac{n}{\left(\frac{9}{7}\right)^{a}} \\
& = & n^{1- \frac{\alpha C \ln \left( \frac{9}{7} \right)}{6} + o(1)} \\
& \to & 0 \textrm{ as } n \to \infty
\textrm{ for } C > \frac{6}{\alpha \ln \left( \frac{9}{7} \right) }.
\phantom{wwwwwwwww}
\qedhere
\end{eqnarray*}
\end{proof}

We now also provide the corresponding lower bound 
for the case when $\liminf \frac{m}{n} > 1$:

\begin{Theorem} \label{maxdeg2}
Let $g \geq 0$ be a constant,
and let $m=m(n)$ satisfy $\liminf \frac{m}{n} >1$.
Then whp
\begin{displaymath}
\Delta (S_{g}(n,m)) = \Omega(\ln n).
\end{displaymath}
\end{Theorem}
\begin{proof}
Let $\alpha > 0$ be as given by Lemma~\ref{main6apps},
let $c \in \left( 0, \frac{1}{\ln \left( \frac{14}{3 \alpha} \right)} \right)$,
and let $\mathcal{G}_{n}$
denote the set of graphs in $\mathcal{S}^{g}(n,m)$
with at least $\alpha n$ 
totally edge-disjoint triangulated appearances of $K_{4}$
and with maximum degree less than $h = \lceil c \ln n \rceil$.
By Lemma~\ref{main6apps},
it will suffice to show
$\frac{|\mathcal{G}_{n}|}{|\mathcal{S}^{g}(n,m)|} \to 0$ as $n \to \infty$.
Consequently,
throughout the remainder of this proof,
we may assume that $n$ is sufficiently large that $h>6$.

Take a graph in $\mathcal{G}_{n}$.
Choose $h+1$ 
\emph{ordered} totally edge-disjoint triangulated appearances of $K_{4}$
(at least $\left( \frac{\alpha n}{2} \right)^{h+1}$ choices for large $n$),
and denote the degree $3$ vertices from each of these (in order)
as $u_{1}, u_{2}, \ldots, u_{h+1}$.
Let us then delete all edges incident to these vertices,
and let us denote the original neighbours of $u_{1}$ 
as $v_{1}, v_{2}$ and $v_{3}$.

Now form a wheel with $u_{1}$ as the central vertex
and with the vertices $u_{2}, u_{3}, \ldots,$ $u_{h+1}$
arranged in clockwise order around it.
Then also join each of $u_{2}, u_{3}, \ldots, u_{h+1}$ to $v_{1}$,
join $u_{2}$ to $v_{2}$ and $v_{3}$,
and join $u_{3}$ to $v_{3}$
(observe that we now have $m$ edges again in total,
and the genus cannot have increased).
Thus,
we find that we can construct at least
$|\mathcal{G}_{n}| \left( \frac{\alpha n}{2} \right)^{h+1}$
(not necessarily distinct) graphs in 
$\mathcal{S}^{g}(n,m)$.
See Figure~\ref{maxdegfig3}.

\begin{figure} [ht]
\setlength{\unitlength}{1cm}
\begin{picture}(10,3.3)(-0.5,-0.4)
\put(0,0){\line(1,0){3}}
\put(0,0){\line(3,2){1.5}}
\put(0,0){\line(3,5){1.5}}
\put(1.5,1){\line(0,1){1.5}}
\put(3,0){\line(-3,2){1.5}}
\put(3,0){\line(-3,5){1.5}}
\put(0,0){\circle*{0.1}}
\put(3,0){\circle*{0.1}}
\put(1.5,1){\circle*{0.1}}
\put(1.5,2.5){\circle*{0.1}}

\put(1.6,1){$u_{1}$}
\put(1.6,2.4){$v_{1}$}
\put(3.1,-0.05){$v_{2}$}
\put(-0.45,-0.05){$v_{3}$}

\put(4,1.25){\vector(1,0){1}}

\put(6,0){\line(1,0){3}}
\put(6,0){\line(1,1){1}}
\put(6,0){\line(3,5){1.5}}
\put(7.5,2.5){\line(0,-1){2}}
\put(9,0){\line(-3,1){1.5}}
\put(9,0){\line(-3,5){1.5}}
\put(6,0){\circle*{0.1}}
\put(9,0){\circle*{0.1}}
\put(7.5,1){\circle*{0.1}}
\put(7.5,2.5){\circle*{0.1}}

\put(7,1){\line(1,-1){0.5}}
\put(7,1){\line(1,1){0.5}}
\put(8,1){\line(-1,-1){0.5}}
\put(8,1){\line(-1,1){0.5}}

\put(7,1){\line(1,0){1}}
\put(6,0){\line(3,1){1.5}}
\put(7,1){\circle*{0.1}}
\put(8,1){\circle*{0.1}}
\put(7.5,0.5){\circle*{0.1}}
\put(7.5,1.5){\circle*{0.1}}
\put(7.5,2.5){\line(-1,-3){0.5}}
\put(7.5,2.5){\line(1,-3){0.5}}
\put(7.5,2.5){\line(1,-2){0.75}}
\put(7.5,1){\oval(1.5,1)[br]}

\put(7.525,1.05){\footnotesize{$u_{1}$}}
\put(7.6,2.4){$v_{1}$}
\put(9.1,-0.05){$v_{2}$}
\put(5.55,-0.05){$v_{3}$}
\put(7.3,0.3){\footnotesize{$u_{2}$}}
\put(6.85,0.7){\footnotesize{$u_{3}$}}

\put(0,0){\line(-1,1){0.4}}
\put(0,0){\line(-1,-1){0.4}}
\put(3,0){\line(1,1){0.4}}
\put(3,0){\line(1,-1){0.4}}
\put(1.5,2.5){\line(-1,1){0.4}}
\put(1.5,2.5){\line(1,1){0.4}}

\put(6,0){\line(-1,1){0.4}}
\put(6,0){\line(-1,-1){0.4}}
\put(9,0){\line(1,1){0.4}}
\put(9,0){\line(1,-1){0.4}}
\put(7.5,2.5){\line(-1,1){0.4}}
\put(7.5,2.5){\line(1,1){0.4}}

\end{picture}
\caption{Using a triangulated appearance of $K_{4}$ 
to construct our new graph.} \label{maxdegfig3}
\end{figure}
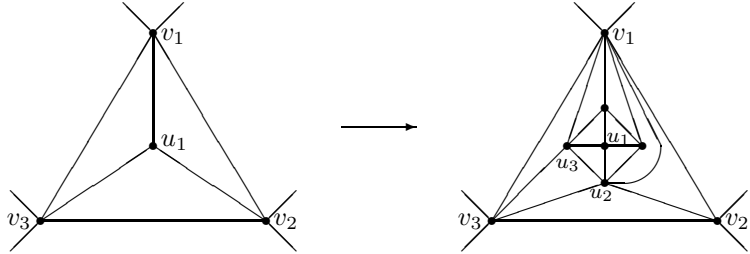

Now let us consider the amount of double-counting.
Note that 
$\deg(u_{1})=h$, $\deg(u_{2})=6$, $\deg(u_{3})=5$,
$\deg(u_{i})=4$ for all $i \geq 4$,
$\deg(v_{1})=h+4$, $\deg(v_{2})=5$, and $\deg(v_{3})=6$,
and recall that all other vertices have degree less than $h$.
Hence,
we can identify $u_{1}$ 
(the only vertex with degree exactly $h$),
after which we can then determine $u_{2}$
(the only neighbour of $u_{1}$ with degree $6$)
and $u_{3}$ (the only neighbour of $u_{1}$ with degree $5$),
and then $u_{4}, u_{5}, \ldots, u_{h+1}$
(using the clockwise ordering).
We then just need to also determine the original neighbours 
of $u_{2}, u_{3}, \ldots, u_{h+1}$
(at most $\left( \frac{7}{3} n \right)^{h}$ possibilities,
as in the previous proof,
since we again need to look 
for what are now triangulated appearances of $K_{3}$).
Thus,
we find that we have built each graph at most
$\left( \frac{7}{3} n \right)^{h}$ times.

Hence,
the number of distinct graphs
(in $\mathcal{S}^{g}(n,m)$)
that we have constructed must be at least
$\frac{|\mathcal{G}_{n}| \left( \frac{\alpha n}{2} \right)^{h+1}}
{\left( \frac{7}{3} n \right)^{h}}
= |\mathcal{G}_{n}| \frac{\alpha}{2} \left( \frac{3 \alpha}{14} \right)^{h} n$,
and so
\begin{eqnarray*}
\phantom{wwwwwwwwq}
\frac{|\mathcal{G}_{n}|}{|\mathcal{S}^{g}(n,m)|}
& \leq & \frac{2}{\alpha} \left( \frac{14}{3 \alpha} \right)^{h} \frac{1}{n} \\
& = & 
n^{-1 + c \ln \left( \frac{14}{3 \alpha} \right) + o(1)} \\
& \to & 0 
\textrm{ as } n \to \infty 
\textrm{ since } c < \frac{1}{\ln \left( \frac{14}{3 \alpha} \right)}. 
\phantom{wwwwwwwwq}
\qedhere
\end{eqnarray*}
\end{proof}

\section{Largest face size: proof of Theorem~\ref{mainface}} \label{facesize}

In this section,
we shall now look at $F(S_{g}(n,m))$,
the size of the largest face of $S_{g}(n,m)$
(maximised over all possible embeddings
with genus at most $g$).
In particular,
we shall see
(as stated already in Theorem~\ref{mainface})
that $F(S_{g}(n,m)) = \Theta (n)$ whp
for $0 < \liminf \frac{m}{n} \leq \limsup \frac{m}{n} < 1$,
and that $F(S_{g}(n,m)) = \Theta (\ln n)$ whp
for $1 < \liminf \frac{m}{n} \leq \limsup \frac{m}{n} < 3$.
A summary of results,
including the cases when $\frac{m}{n} \to 1$ and $\frac{m}{n} \to 3$,
is given in Table~\ref{facetab}.

\begin{table} [ht]
\centering
\caption{
A summary of the size of the largest face of $S_{g}(n,m)$.
}
\label{facetab}
{\renewcommand{\arraystretch}{1.5}
\begin{tabular}{|c|c||c|}
\hline
\multicolumn{2}{|c||}{
Range of $m$
}
&
$F(S_{g}(n,m))$ \\
\hline
\hline
\multicolumn{2}{|c||}{
$0 < \liminf \frac{m}{n} \leq \limsup \frac{m}{n} < 1$
} &
$\Theta (n)$ \textrm{whp} (Thm.~\ref{face0to1})
\\
\hline
&
$0 < n-m = \Omega (n^{3/5})$ &
$\Omega (n-m)$ whp (Thm.~\ref{faceat1}) \\
\cline{2-3}
&
$|m-n| = O(n^{3/5})$ &
$\Omega (n^{3/5})$ whp (Thm.~\ref{faceat1})
\\
\cline{2-3}
$\frac{m}{n} \to 1$ &
$0 < m-n =
\left\{ \begin{array}{ll}
\vspace{0.1cm}
\Omega (n^{3/5}) \\
o \left( \frac{n}{(\log n)^{2/3}} \right) 
\end{array} \right.$ &
$\Omega \left( \left( \frac{n}{m-n} \right) ^{3/2} \right)$ whp 
(Thm.~\ref{faceat1}) \\
\cline{2-3}
&
$m-n = \Omega \left( \frac{n}{(\log n)^{2/3}} \right)$ &
unknown \\
\hline
\multicolumn{2}{|c||}{
$1 < \liminf \frac{m}{n} \leq \limsup \frac{m}{n} < 3$ 
} &
$\Theta (\ln n)$ \textrm{whp} (Thm.~\ref{face1to3} \&~\ref{faceat3}) \\
\hline
\multicolumn{2}{|c||}{
$\frac{m}{n} \to 3$
} &
$O(\ln n)$ \textrm{whp} (Thm.~\ref{faceat3}) \\
\hline
\end{tabular}}
\end{table}

We start with the region
$0 < \liminf \frac{m}{n} \leq \limsup \frac{m}{n} < 1$.
Here,
we know from Theorem~\ref{ntree}
that whp there are linearly many components
isomorphic to any given tree,
and so
(since these can all be placed in the same face,
and since no face can ever exceed linear size)
we obtain our result immediately:

\begin{Theorem} \label{face0to1}
Let $g \geq 0$ be a constant,
and let $m=m(n)$ satisfy 
$0 < \liminf \frac{m}{n} \leq \limsup \frac{m}{n} < 1$.
Then whp
\begin{displaymath}
F(S_{g}(n,m)) = \Theta (n).
\end{displaymath}
\end{Theorem}
\qed

Similarly,
we may use Lemma~\ref{kmslemma}
on the number of edges in trees and unicyclic components
to obtain lower bounds for the largest face size when $\frac{m}{n} \to 1$:

\begin{Theorem} \label{faceat1}
Let $g \geq 0$ be a constant,
and let $m=m(n)$ satisfy 
$\frac{m}{n} \to 1$.
Then whp
\begin{eqnarray*}
F(S_{g}(n,m)) =
\left\{ 
{\renewcommand{\arraystretch}{1.5}
\begin{array}{lll}
\Omega (n-m)
& \textrm{for } 
0 < n-m = \Omega (n^{3/5}) \\
\Omega (n^{3/5})
& \textrm{for } 
|m-n| = O(n^{3/5}) \\
\Omega \left( \left( \frac{n}{m-n} \right) ^{3/2} \right)
& \textrm{for } 
0 < m-n =
\left\{ \begin{array}{ll}
\Omega (n^{3/5}) \\
o \left( \frac{n}{(\log n)^{2/3}} \right).
\end{array} \right. \\
\end{array}} \right.
\end{eqnarray*} 
\end{Theorem}
\qed

For the random graph $S_{g}(n)$,
it is shown in~\cite{mcdr}
that whp the largest face size is $\Theta (\ln n)$.
We shall now see that this also holds for our graph $S_{g}(n,m)$
for the region when
$1 < \liminf \frac{m}{n} \leq \limsup \frac{m}{n} < 3$.
We start with the lower bound:

\begin{Theorem} \label{face1to3}
Let $g \geq 0$ be a constant,
and let $m=m(n)$ satisfy 
$1 < \liminf \frac{m}{n} \leq \limsup \frac{m}{n} < 3$.
Then whp
\begin{displaymath}
F(S_{g}(n,m)) = \Omega (\ln n).
\end{displaymath}
\end{Theorem}
\begin{proof}
We follow the method of proof of Theorem 3.1 of~\cite{mcdr},
noting that it suffices to show that whp $S_{g}(n,m)$ contains a pendant copy of a path with $\Omega (\ln n)$ vertices.

Let $\alpha$ be as given by Lemma~\ref{pendant},
let $c \in \left(0, \frac{1}{\ln \left( \frac{2}{\alpha} \right)} \right)$,
let $h = \lceil c \ln n \rceil$,
and let $\mathcal{G}_{n}$ denote the set of graphs in $\mathcal{S}^{g}(n,m)$
with~(i) at least $\alpha n$ pendant edges
and~(ii) no pendant copy of $P_{h}$
(where $P_{h}$ denotes a path with $h$ vertices).
By Lemma~\ref{pendant},
it will suffice to show
$\frac{|\mathcal{G}_{n}|}{|\mathcal{S}^{g}(n,m)|} \to 0$ as $n \to \infty$.

Take a graph in $\mathcal{G}_{n}$.
Let us choose $h$ \emph{ordered} pendant edges
(at least $\left( \frac{\alpha n}{2} \right)^{h}$ choices for large $n$),
and let us denote the ordered chosen pendant vertices as $v_{1}, v_{2}, \ldots, v_{h}$.
Delete all $h-1$ (pendant) edges incident to $v_{2}, v_{3}, \ldots, v_{h}$,
and insert edges 
$v_{1}v_{2},v_{2}v_{3}, \ldots, v_{h-1}v_{h}$
to create a pendant copy of $P_{h}$.
Thus, we find that we can construct at least
$|\mathcal{G}_{n}| \left( \frac{\alpha n}{2} \right)^{h}$
(not necessarily distinct) graphs in $\mathcal{S}^{g}(n,m)$.

Now let us consider the amount of double-counting.
Firstly, note that if a new pendant copy of a graph $H$ is produced by deleting an edge $uv$,
then either $u$ or $v$ (or both) must belong to this pendant copy.
Thus, if $v$ was originally a pendant vertex,
then it must be that $u$ belongs to this new pendant copy.
Secondly, let us call a pendant copy of the path $P_{h}$ `straight'
if it is joined to the rest of the graph at an \emph{end-point} of the path,
and let us note that any vertex can only ever be in at most two straight pendant copies of $P_{h}$
(at most one in each direction,
possibly both ways if an entire component is a path).
Hence, each time we deleted a pendant edge,
we can have only increased the number of straight pendant copies of $P_{h}$ by at most two.
Similarly, when we inserted the path $v_{1} v_{2} \ldots v_{h}$,
we can also have only increased the number of straight pendant copies of $P_{h}$ by at most two
(since any new pendant copy of $P_{h}$ would contain $v_{1}$).

Thus, given one of our constructed graphs,
there will be at most $2h$ straight pendant copies of $P_{h}$.
Hence, there are at most $2h$ possibilities for $v_{h}$,
after which we can then determine the ordered vertices $v_{h-1}, v_{h-2}, \ldots, v_{1}$.
We then also need to determine the original neighbours of $v_{2}, v_{3}, \ldots, v_{h}$ (at most $n^{h-1}$ possibilities).
Thus, we find that we have built each graph at most $2h n^{h-1}$ times.

Hence, the number of distinct graphs (in $\mathcal{S}^{g}(n,m)$) that we have constructed must be at least
$\frac{|\mathcal{G}_{n}| \left( \frac{\alpha n}{2} \right)^{h}}{2h n^{h-1}}
= |\mathcal{G}_{n}| \frac{1}{2h} \left( \frac{\alpha}{2} \right)^{h} n$,
and so
\begin{eqnarray*}
\phantom{wwwwwwwww}
\frac{|\mathcal{G}_{n}|}{|\mathcal{S}^{g}(n,m)|} & \leq & 2h \left( \frac{2}{\alpha} \right)^{h} \frac{1}{n} \\
& = & 
n^{-1 + c \ln \left( \frac{2}{\alpha} \right) + o(1)} \\
& \to & 0 
\textrm{ as } n \to \infty 
\textrm{ since } c < \frac{1}{\ln \left( \frac{2}{\alpha} \right)}. 
\phantom{wwwwwwwww}
\qedhere
\end{eqnarray*}
\end{proof}

We also now provide a corresponding upper bound:

\begin{Theorem} \label{faceat3}
Let $g \geq 0$ be a constant,
and let $m=m(n)$ satisfy $\liminf \frac{m}{n} > 1$.
Then whp
\begin{displaymath}
F(S_{g}(n,m)) = O(\ln n).
\end{displaymath}
\end{Theorem}
\begin{proof}
First,
note that the size of a face with $k$ vertices and $l$ edges
is at most $2l$,
and hence at most $2(3k-6+6g)$.
Thus,
it will suffice for us to show that 
the largest number of \emph{vertices} in any face of $S_{g}(n,m)$
is $O(\ln n)$ with high probability.

Let $\alpha >0$ be as given by Lemma~\ref{main6apps},
let $C > \frac{6}{\alpha \ln \left( \frac{9}{7} \right)}$,
and let $\mathcal{G}_{n}$ denote the set of graphs in $\mathcal{S}^{g}(n,m)$
with at least $\alpha n$ 
totally edge-disjoint triangulated appearances of $K_{4}$
and with a face (in some embedding with genus at most $g$)
containing at least $C \ln n$ vertices.
By Lemma~\ref{main6apps},
it will suffice to show
$\frac{|\mathcal{G}_{n}|}{|\mathcal{S}^{g}(n,m)|} \to 0$ as $n \to \infty$.

Before we continue,
for each graph in $\mathcal{G}_{n}$
let us fix a particular embedding
from among those which maximise the number of vertices in a single face
(over all embeddings with genus at most $g$),
let us fix a particular face 
from among those with the largest number of vertices in this embedding,
and let us fix a particular `clockwise' ordering
of the vertices in this face
(this should be done by placing an imaginary vertex at some point inside the face,
inserting exactly one edge from here to every vertex in the face
in such a way that no crossing edges are introduced,
and then taking a particular clockwise ordering 
in terms of how these edges leave the imaginary vertex).

Now let us take one of our graphs,
and recall that the number of vertices in our chosen face is $d \geq C \ln n$.
Let us denote these vertices,
using our given ordering,
as $v_{1}, v_{2}, \ldots, v_{d}$.

Let $a = \lfloor \frac{\alpha C \ln n}{6} \rfloor$,
and let us choose $a$ of these vertices
(at least $\left( ^{d} _{a} \right) 
\geq \left( \frac{d}{a} \right)^{a}
\geq \left( \frac{6}{\alpha} \right)^{a}$ choices).
Let us denote the chosen vertices 
in clockwise order as $v_{i_{1}}, v_{i_{2}}, \ldots,$ $v_{i_{a}}$,
where $v_{i_{1}}$ is such that 
$v_{d} \in \{ v_{i_{1}}, v_{i_{1}+1}, \ldots, v_{i_{2}-1} \}$.

Let us also choose $a+1$ 
\emph{ordered} totally edge-disjoint triangulated appearances of $K_{4}$ 
that do not contain any of these chosen vertices as part of the $K_{4}$
(at least $\left( \frac{\alpha n}{2} \right)^{a+1}$ choices for large $n$).
For these chosen triangulated appearances,
let us denote (in order) the $a$ degree $3$ vertices as
$u_{0}, u_{1}, u_{2}, \ldots, u_{a}$.

Now delete all $3(a+1)$ edges incident to
$u_{0}, u_{1}, u_{2}, \ldots, u_{a}$,
and form a wheel with $u_{0}$ as the central vertex
and with the vertices $u_{1}, u_{2}, \ldots, u_{a}$
arranged in clockwise order around it.
For all $j \in \{ 1,2, \ldots, a \}$,
join $u_{j}$ to $v_{i_{j}}$,
and let us then also join 
$u_{1}$ to $v_{i_{2}}$, $u_{2}$ to $v_{i_{3}}$, and $u_{3}$ to $v_{i_{4}}$
(observe that we now have $m$ edges again in total,
and the genus cannot have increased).
Thus,
we find that we can construct at least
$|\mathcal{G}_{n}| \left( \frac{6}{\alpha} \right)^{a}
\left( \frac{\alpha n}{2} \right)^{a+1}
= |\mathcal{G}_{n}| (3n)^{a+1} \frac{\alpha}{6}$
(not necessarily distinct) graphs in $\mathcal{S}^{g}(n,m)$.
See Figure~\ref{facesizefig}.

\begin{figure} [ht]
\setlength{\unitlength}{1cm}
\begin{picture}(10,2.4)(-0.5,-0.2)
\put(1.5,1){\oval(3,2)}
\put(0,1.5){\circle*{0.1}}
\put(0,0.5){\circle*{0.1}}
\put(1,2){\circle*{0.1}}
\put(3,1){\circle*{0.1}}
\put(3,0.5){\circle*{0.1}}
\put(3,1.5){\circle*{0.1}}
\put(1,0){\circle*{0.1}}
\put(2,0){\circle*{0.1}}
\put(2,2){\circle*{0.1}}
\put(1.8,2.2){$v_{i_{1}}$}
\put(3,1.6){$v_{i_{2}}$}
\put(2,-0.2){$v_{i_{3}}$}
\put(-0.4,0.4){$v_{i_{4}}$}

\put(4,1){\vector(1,0){1}}

\put(7.5,1){\oval(3,2)}
\put(7.5,1){\circle{1}}
\put(6,1.5){\circle*{0.1}}
\put(6,0.5){\circle*{0.1}}
\put(7,2){\circle*{0.1}}
\put(9,1){\circle*{0.1}}
\put(9,0.5){\circle*{0.1}}
\put(9,1.5){\circle*{0.1}}
\put(7.5,1){\circle*{0.1}}
\put(7.5,0.5){\circle*{0.1}}
\put(7.5,1.5){\circle*{0.1}}
\put(7,1){\circle*{0.1}}
\put(8,1){\circle*{0.1}}
\put(7,0){\circle*{0.1}}
\put(8,0){\circle*{0.1}}
\put(8,2){\circle*{0.1}}
\put(7,1){\line(1,0){1}}
\put(7.5,0.5){\line(0,1){1}}
\put(7.5,1.5){\line(1,0){1.5}}
\put(8,0){\line(0,1){1}}
\put(6,0.5){\line(2,1){1}}
\put(6,0.5){\line(1,0){1.5}}
\put(8,1){\line(2,1){1}}
\put(8,0){\line(-1,1){0.5}}
\put(8,2){\line(-1,-1){0.5}}

\put(7.8,2.2){$v_{i_{1}}$}
\put(9,1.6){$v_{i_{2}}$}
\put(8,-0.2){$v_{i_{3}}$}
\put(5.6,0.4){$v_{i_{4}}$}

\put(7.5,1.05){$u_{0}$}
\put(7.1,1.55){$u_{1}$}
\put(6.6,1.1){$u_{4}$}
\put(7.2,0.3){$u_{3}$}
\put(8,0.8){$u_{2}$}
\end{picture}
\caption{Using a large face to construct our new graph.} \label{facesizefig}
\end{figure}
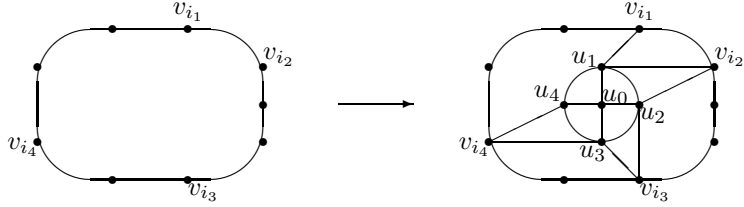

Now let us consider the amount of double-counting.
We need to first identify $u_{0}$
(at most $n$ possibilities),
after which we can then determine 
the \emph{unordered} set
$\{ u_{1}, u_{2}, \ldots, u_{a} \}$
as being the neighbours of $u_{0}$.
We then also need to determine the original neighbours
of $u_{0}, u_{1}, u_{2}, \ldots, u_{a}$
(at most $\left( \frac{7}{3} n \right)^{a+1}$ possibilities,
as in the proof of Theorem~\ref{maxdeg},
since we need to look for triangulated appearances of $K_{3}$).
After this,
we then know the original graph,
and hence the original embedding,
our chosen face,
and the order of $v_{1}, v_{2}, \ldots, v_{d}$.
Hence,
we can then determine the order of $u_{1}, u_{2}, \ldots, u_{a}$.
Thus,
we find that we have built each graph at most
$\left( \frac{7}{3} n \right)^{a+1} n$ times.

Hence,
the number of distinct graphs (in $\mathcal{S}^{g}(n,m)$)
that we have constructed
must be at least
$\frac{|\mathcal{G}_{n}| (3n)^{a+1} \frac{\alpha}{6}}
{\left( \frac{7}{3} n \right)^{a+1} n}
= |\mathcal{G}_{n}| \frac{\alpha}{6n} \left( \frac{9}{7} \right)^{a+1}$,
and so 
\begin{eqnarray*}
\phantom{wwwwwwwww}
\frac{|\mathcal{G}_{n}|}{|\mathcal{S}^{g}(n,m)|} 
& \leq & \frac{6}{\alpha} \frac{n}{\left( \frac{9}{7} \right)^{a+1}} \\
& = & 
n^{1- \frac{\alpha C \ln \left( \frac{9}{7} \right)}{6} + o(1)} \\
& \to & 0 
\textrm{ as } n \to \infty 
\textrm{ since } C > \frac{6}{\alpha \ln \left( \frac{9}{7} \right)}.
\phantom{wwwwwwwww}
\qedhere
\end{eqnarray*}
\end{proof}

\section{Appearances: proof of Lemma~\ref{mainapps}} \label{appsection}

In this section,
we shall investigate the number of appearances in $S_{g}(n,m)$
of given subgraphs (see Definition~\ref{defapps}).
The main feature here will be a proof of Lemma~\ref{mainapps},
but we will also derive
(in Theorem~\ref{newuniform})
a new result on the uniform convergence of $|\mathcal{S}^{g}(n,m)|$
to the relevant growth constant.

We start with a lemma that defines the growth constant function $\gamma (q)$:

\begin{Lemma} [\cite{chap}, Theorem 1.1]
There exists a continuous function $\gamma (q)$ such that,
given any constants $g \geq 0$ and $q \in (1,3)$,
we have
\begin{displaymath}
\left( \frac{|\mathcal{S}^{g}(n,\lfloor qn \rfloor)|}{n!} \right)^{1/n} 
\to \gamma (q) \textrm{ as } n \to \infty.
\end{displaymath} 
For all $q \in (1,3)$,
we have $e < \gamma (q) \leq \gamma_{l} \approx 27.23$,
where $\gamma_{l}$ is the labelled planar graph growth constant.
\end{Lemma}

For planar graphs,
the following useful uniform convergence result is known:

\begin{Lemma} [\cite{ger}, Lemma 2.9] \label{uniform}
Let $a \in (1,3)$ and $\eta>0$ be constants.
Then there exists $n_{0}$ such that, 
for all $n \geq n_{0}$ and all $m \in [an,3n-6]$,
we have
\begin{displaymath}
\left| \left( \frac{|\mathcal{S}^{0}(n,m)|}{n!} \right)^{1/n} - 
\gamma \left( \frac{m}{n} \right) \right| < \eta.
\end{displaymath} 
\end{Lemma}

As mentioned,
we shall later (in Theorem~\ref{newuniform}) generalise Lemma~\ref{uniform}
to non-zero genus.
However,
we first come to the main work of this section,
where we provide a proof of Lemma~\ref{mainapps},
showing that whp $S_{g}(n,m)$
will have linearly many appearances of any given planar graph.
As the full proof is quite long,
we also give a sketch of the proof:

\begin{appsketch}
Recall that 
$\left( \frac{| \mathcal{S}^{g} (n, \lfloor qn \rfloor) |}{n!} \right)^{1/n}
\to \gamma (q)$ for $q \in (1,3)$.
Although we have not yet shown that this convergence is uniform,
we can still certainly choose any large but \emph{finite} number of values
$q_{1}^{*}, q_{2}^{*}, \ldots, q_{T}^{*}$
and then find an $N$ such that
\begin{displaymath}
(1 - \epsilon)^{n} n! (\gamma (q_{i}^{*}))^{n} 
\leq |\mathcal{S}^{g} (n, \lfloor q_{i}^{*}n \rfloor)|
\leq (1 + \epsilon)^{n} n! (\gamma (q_{i}^{*}))^{n}
\end{displaymath}
for all \emph{these} $q_{i}^{*}$ for all $n \geq N$ (for a given $\epsilon > 0$).
The upper bound here will be of particular importance to us.

We then suppose
(aiming for a contradiction)
that the statement of the theorem is false for some $n=k$,
and we find a value $q_{j}^{*}$ close to $\frac{m(k)}{k}$.
Using $|\mathcal{S}^{g}(k,m(k))| \geq |\mathcal{S}^{0}(k,m(k))|$,
Lemma~\ref{uniform},
and the continuity of $\gamma (q)$,
we may then obtain
$|\mathcal{S}^{g}(k,m(k))| \geq (1 - \epsilon)^{k} k! (\gamma (q_{j}^{*}))^{k}$.

We then take graphs in $|\mathcal{S}^{g}(k,m(k))|$
without $\alpha k$ vertex-disjoint appearances of $H$,
and to each of these 
we attach many appearances of carefully selected graphs $H_{1}$ and $H_{2}$,
which both contain appearances of $H$.
By choosing $H_{1}$ and $H_{2}$
to have the appropriate ratio of edges to vertices,
we may consequently construct many graphs
in $\mathcal{S}^{g} ((1+\delta)k, \lfloor q_{j}^{*} (1+\delta)k \rfloor)$
for some $\delta > 0$.

The fact that the original graphs were assumed to contain few appearances of $H$
is then used to bound the amount of double-counting,
and so we find that we obtain a contradiction to our earlier upper bound
on $|\mathcal{S}^{g} (n, \lfloor q_{j}^{*}n \rfloor)|$ when $n=(1+\delta)k$.
\end{appsketch}
\begin{appfull}
Let $b>1$ denote $\liminf \frac{m}{n}$,
let $B<3$ denote $\limsup \frac{m}{n}$,
and let $h$ denote $|H|$.
Let $l \geq 3$ then be an integer chosen to satisfy both
\begin{equation}
\frac{e(H)+l+1}{h+l} < b \label{l1}
\end{equation}
and 
\begin{equation}
\frac{e(H)+3l-4}{h+l} > B, \label{l2}
\end{equation}
let 
\begin{equation}
\beta = e^{2} ( \gamma_{l} )^{h+l} 2(2h+l+1) (h+l)!, \label{beta}
\end{equation}
and let $\alpha$ be a fixed constant in $\left( 0, \frac{1}{\beta} \right)$.
Since $\alpha \beta <1$,
we may then also choose a value $\epsilon \in \left( 0, \frac{1}{3} \right)$
such that 
\begin{equation}
(\alpha \beta)^{\alpha} = 1 - 3 \epsilon. \label{epsilon}
\end{equation}

By continuity of $\gamma (q)$,
it is possible to find $d>0$
such that $|\gamma (q_{1}) - \gamma (q_{2})| < \frac{\epsilon}{2}$
whenever $q_{1}, q_{2} \in [b,B]$ and $|q_{1}-q_{2}| \leq d$.
We may assume that $d$ is small enough that is also satisfies both
\begin{equation}
\frac{\alpha (e(H)+l+1) + d}{\alpha (h+l)} \leq b \label{b}
\end{equation}
and 
\begin{equation}
\frac{\alpha (e(H)+3l-4) - d}{\alpha (h+l)} \geq B. \label{B}
\end{equation}
Let us then split $[b,B]$ into a finite number of intervals of length at most $d$,
and for each interval $i$
let us select a value $q_{i}^{*}$ in that interval.

Let $N$ then be chosen so that
$|\mathcal{S}^{g}(n, \lfloor q_{i}^{*} n \rfloor)| 
\leq (1 + \epsilon)^{n}n! (\gamma (q_{i}^{*}))^{n}$ 
for all $i$ for all $n \geq N$,
and let us suppose that the statement of the theorem doesn't hold 
for some $k \geq N$
(throughout the remainder of this proof,
we will assume that $N$ is large enough that various inequalities involving 
$l, \beta, \alpha, \epsilon$ and $k$ are satisfied).
Let $M$ denote $m(k)$,
and let $\mathcal{G}_{k}$ denote the set of graphs in $\mathcal{S}^{g}(k,M)$
which contain at most $\alpha k$ appearances of $H$
(so $|\mathcal{G}_{k}| \geq 
e^{- \alpha k} |\mathcal{S}^{g}(k,M)|$).

Without loss of generality,
suppose $\frac{M}{k}$ is in interval $j$ of our subdivision of $[b,B]$.
Thus,
we have 
\begin{eqnarray*}
|\mathcal{S}^{g}(k,M)| 
& \geq & |\mathcal{S}^{0}(k,M)| \\
& \geq & \left(1- \frac{\epsilon}{2} \right)^{k} 
\left( \gamma \left( \frac{M}{k} \right) \right)^{k} k!
\textrm{ by Lemma~\ref{uniform} (for large $k$)} \\
& \geq & \left( 1- \frac{\epsilon}{2} \right)^{k} 
\left( \gamma (q_{j}^{*}) - \frac{\epsilon}{2} \right)^{k} k! \\
& \geq & (1- \epsilon)^{k} (\gamma (q_{j}^{*}))^{k} k!
\textrm{ (since $\gamma (q) > 1$ for all $q$)},
\end{eqnarray*}
and so
$|\mathcal{G}_{k}| \geq 
e^{- \alpha k} (1- \epsilon)^{k} (\gamma (q_{j}^{*}))^{k} k!$.

Recall
\begin{displaymath}
\frac{\alpha (e(H)+l+1) + d}{\alpha (h+l)} 
\stackrel{\eqref{b}}{\leq} b 
\leq q_{j}^{*}
\end{displaymath}
and 
\begin{displaymath}
\frac{\alpha (e(H)+3l-4) - d}{\alpha (h+l)} 
\stackrel{\eqref{B}}{\geq} B 
\geq q_{j}^{*}.
\end{displaymath}
Hence,
\begin{displaymath}
\frac{q_{j}^{*} + \alpha (e(H)+l+1) + d}{1 + \alpha (h+l)} 
\leq q_{j}^{*}
\leq \frac{q_{j}^{*} + \alpha (e(H)+3l-4) - d}{1 + \alpha (h+l)},
\end{displaymath}
so
\begin{displaymath}
\frac{(q_{j}^{*}+d)k + \alpha k (e(H)+l+1)}
{k + \alpha k (h+l)} 
\leq q_{j}^{*}
\leq \frac{(q_{j}^{*}-d)k + \alpha k (e(H)+3l-4)}
{k + \alpha k (h+l)},
\end{displaymath}
and so
\begin{displaymath}
\frac{(q_{j}^{*}+d)k + \lceil \alpha k \rceil (e(H)+l+1)}
{k + \lceil \alpha k \rceil (h+l)} 
\stackrel{\eqref{l1}}{\leq} q_{j}^{*}
\stackrel{\eqref{l2}}{\leq} 
\frac{(q_{j}^{*}-d)k + \lceil \alpha k \rceil (e(H)+3l-4)}
{k + \lceil \alpha k \rceil (h+l)}.
\end{displaymath}
Thus,
since $q_{j}^{*}-d \leq \frac{M}{k} \leq q_{j}^{*}+d$,
we have
\begin{displaymath}
\frac{M+ \lceil \alpha k \rceil (e(H)+l+1)}{k+\lceil \alpha k \rceil (h+l)} 
\leq q_{j}^{*} \leq
\frac{M+ \lceil \alpha k \rceil (e(H)+3l-4)}{k+\lceil \alpha k \rceil (h+l)}.
\end{displaymath}

Hence,
we can find integers $r^{-}$ and $r^{+}$ in $\{l-1,l,\ldots,3l-6\}$
satisfying $r^{+}=r^{-}+1$ and
\begin{displaymath}
\frac{M+ \lceil \alpha k \rceil (e(H)+2+r^{-})}{k+\lceil \alpha k \rceil (h+l)} 
\leq q_{j}^{*} \leq
\frac{M+ \lceil \alpha k \rceil (e(H)+2+r^{+})}{k+\lceil \alpha k \rceil (h+l)}.
\end{displaymath}
Since $r^{-},r^{+} \in \{l-1,l,\ldots,3l-6\}$,
we can thus find 
connected planar graphs $H^{-}$ and $H^{+}$ on $\{1,2,\ldots,l\}$
with $e(H^{-})=r^{-}$
and $e(H^{+})=r^{+}=r^{-}+1$.

Now let $H^{\prime}$ be an order-preserving copy of $H$ 
on $\{l+1,l+2,\ldots,l+h\}$,
let $H_{1}$ be formed from $H^{-}$ and $H^{\prime}$
by adding an edge between vertex $l$ and vertex $l+1$,
and let $H_{2}$ be the analogous graph formed from $H^{+}$ and $H^{\prime}$.
Note that $H_{1}$ and $H_{2}$
(and indeed any appearances of $H_{1}$ and $H_{2}$)
both contain appearances of $H$.
Also,
note that we have 
\begin{displaymath}
\frac{M+ \lceil \alpha k \rceil (e(H_{1})+1)}{k+\lceil \alpha k \rceil (h+l)} 
\leq q_{j}^{*} \leq
\frac{M+ \lceil \alpha k \rceil (e(H_{2})+1)}{k+\lceil \alpha k \rceil (h+l)},
\end{displaymath}
and hence 
\begin{displaymath}
M + \lceil \alpha k \rceil (e(H_{1})+1)
\leq \lfloor q_{j}^{*} (k + \lceil \alpha k \rceil (h+l)) \rfloor
\leq M + \lceil \alpha k \rceil (e(H_{2})+1)
\end{displaymath}
(using the integrality of 
$M + \lceil \alpha k \rceil (e(H_{1})+1)$
to obtain the left-hand inequality).

Now let 
\begin{equation}
\delta = \frac{ \lceil \alpha k \rceil (h+l) }{k}. \label{delta}
\end{equation}
Starting with graphs in $\mathcal{G}_{k}$,
we shall construct graphs 
in $\mathcal{S}^{g}((1+\delta)k, \lfloor q_{j}^{*}(1+\delta)k \rfloor)$
by attaching 
$k_{1}$ appearances of $H_{1}$ 
and $k_{2} = \lceil \alpha k \rceil - k_{1}$ appearances of $H_{2}$.

Note that we shall need to achieve the correct balance of $k_{1}$ and $k_{2}$
so that our constructed graphs will indeed have 
$\lfloor q_{j}^{*}(1+\delta)k \rfloor) 
= \lfloor q_{j}^{*} ( k + \lceil \alpha k \rceil ) (h+l) \rfloor$ edges.
But observe that (since $e(H_{2})=e(H_{1})+1$)
$k_{1}$ and $k_{2}$ can be chosen so that
the number of edges in our constructed graph is any desired integer from
$M + \lceil \alpha k \rceil (e(H_{1})+1)$ to
$M + \lceil \alpha k \rceil (e(H_{2})+1)$,
and so this is okay. \\

Having obtained the appropriate values of $k_{1}$ and $k_{2}$,
let us now construct our graphs 
in $\mathcal{S}^{g}((1+\delta)k, \lfloor q_{j}^{*}(1+\delta)k \rfloor)$:

Choose $\delta k$ special vertices 
(we have $\left( ^{(1+\delta)k} _{\phantom{qq} \delta k} \right)$ 
choices for these),
and partition them into $\lceil \alpha k \rceil$ unordered blocks of size $h+l$
(we have 
$\left( ^{\phantom{wwq}\delta k}_{h+l, \ldots, h+l} \right) 
\frac{1}{\left \lceil \alpha k \right \rceil !}$ 
choices for this).
Divide the blocks into two sets of size $k_{1}$ and $k_{2}$.
On each of the first $k_{1}$ blocks,
we put a copy of $H_{1}$ 
such that the increasing bijection from $V(H_{1})$ to the block 
is an isomorphism between $H_{1}$ and this copy.
We do the same for the set of $k_{2}$ blocks, 
except with $H_{2}$ instead of $H_{1}$.

On the remaining (i.e.~non-special) vertices,
choose a graph $G \in \mathcal{G}_{k}$.
We may then attach our copies of $H_{1}$ and $H_{2}$ 
to any vertices in $V(G)$
($k^{\lceil \alpha k \rceil}$ choices)
to create appearances of $H_{1}$ and $H_{2}$.

Thus, for each choice of special vertices and each choice of $G$, 
the number of graphs in
$\mathcal{S}^{g}((1+\delta)k, \lfloor q_{j}^{*}(1+\delta)k \rfloor)$
that we may construct is at least
\begin{eqnarray*}
\left( ^{\phantom{wwq}\delta k}_{h+l, \ldots, h+l} \right)
\frac{1}{\lceil \alpha k \rceil !} k^{\lceil \alpha k \rceil} 
& = & \frac{(\delta k)! k^{\lceil \alpha k \rceil}}
{((h+l)!)^{\lceil \alpha k \rceil} \lceil \alpha k \rceil !} \\
& \geq & \frac{(\delta k)!} {((h+l)! \alpha)^{\lceil \alpha k \rceil}} 
\end{eqnarray*}
(for $k$ large enough that 
$\lceil \alpha k \rceil ! \leq (\alpha k)^{\lceil \alpha k \rceil}$).

Hence,
we may construct at least 
\begin{displaymath}
\left( ^{(1+\delta)k}_{\phantom{qq} \delta k} \right) 
e^{- \alpha k} (1 - \epsilon)^{k} (\gamma(q_{j}^{*}))^{k} k! 
\frac{(\delta k)!}{((h+l)! \alpha)^{\lceil \alpha k \rceil}}
\end{displaymath}
(not necessarily distinct) graphs 
in $\mathcal{S}^{g}((1+\delta)k, \lfloor q_{j}^{*}(1+\delta)k \rfloor)$ 
in total. \\

We shall now consider the amount of double-counting:

Recall that $G$ did not contain
$\alpha k$ vertex-disjoint appearances of $H$,
and so contained fewer than $h \alpha k$ appearances of $H$ in total
(by Observation~\ref{appinter}).
Recall also that each appearance of $H_{1}$ 
contains an appearance of $H$,
and so $G$ contained fewer than $h \alpha k$ appearances of $H_{1}$.
When we deliberately attach an appearance of $H_{1}$ or $H_{2}$,
the number of `accidental' appearances of $H_{1}$ 
that we create in the graph
will be at most $h+l$
(considering the number of cut-edges).
Thus, our created graph will have at most
$(2h+l+1) \lceil \alpha k \rceil$
appearances of $H_{1}$.
Similarly,
our created graph will have at most
$(2h+l+1) \lceil \alpha k \rceil$
appearances of $H_{2}$.

Let $x=2(2h+l+1)$.
Then, given one of our constructed graphs, we have at most
$\left( ^{x \lceil \alpha k \rceil}_{\phantom{i}\lceil \alpha k \rceil} \right)
\leq (xe)^{\lceil \alpha k \rceil}$
choices for which were the special vertices.
Once we have identified these, we then know what $G$ was.
Thus, each graph is constructed at most $(xe)^{\lceil \alpha k \rceil}$ times. \\

Therefore, 
we find that
the number of distinct graphs that we have created in
$\mathcal{S}^{g}((1+\delta)k, \lfloor q_{j}^{*}(1+\delta)k \rfloor)$
is at least 
\begin{eqnarray*}
& & \left( ^{(1+\delta)k}_{\phantom{qq} \delta k} \right) 
e^{- \alpha k} (1 - \epsilon)^{k} (\gamma(q_{j}^{*}))^{k} k! 
\frac{(\delta k)!}{((h+l)! \alpha)^{\lceil \alpha k \rceil}} 
(xe)^{- \lceil \alpha k \rceil} \\
& \stackrel{\eqref{delta}}{\geq} & 
((1+\delta)k)! (\gamma(q_{j}^{*}))^{(1+\delta)k} (1-\epsilon)^{k} 
\left( e^{2} (\gamma(q_{j}^{*}))^{(h+l)}x(h+l)! \alpha \right) 
^{- \lceil \alpha k \rceil} \\
& \stackrel{\eqref{beta}}{\geq} & 
((1+\delta)k)! (\gamma(q_{j}^{*}))^{(1+\delta)k} (1-\epsilon)^{k} 
(\alpha \beta)^{- \lceil \alpha k \rceil} \\
& \stackrel{\eqref{epsilon}}{\geq} & 
|\mathcal{S}^{g}((1+\delta)k, \lfloor q_{j}^{*}(1+\delta)k \rfloor)| 
(1+\epsilon)^{-(1+\delta)k} (1-\epsilon)^{k} (1-3\epsilon)^{-k} \\
& \geq & 
|\mathcal{S}^{g}((1+\delta)k, \lfloor q_{j}^{*}(1+\delta)k \rfloor)| 
\left( \frac{(1-\epsilon)}{(1-3\epsilon)(1+\epsilon)^{2}} \right) ^{k} \\
& & \textrm{(since we may assume $k$ is large enough that } \delta<1) \\
& > & |\mathcal{S}^{g}((1+\delta)k, \lfloor q_{j}^{*}(1+\delta)k \rfloor)| \\
& & \textrm{(since } 
(1-3\epsilon)(1+\epsilon)^{2}=1-\epsilon -5\epsilon^{2} - 3\epsilon^{3}).
\end{eqnarray*}
Thus,
we have obtained our desired contradiction. \qed 
\end{appfull}

We now look to conclude this section by using Lemma~\ref{mainapps}
to obtain a uniform convergence result similar to Lemma~\ref{uniform}
for arbitrary genus.
Recall that we have already utilised Lemma~\ref{mainapps} 
to produce Lemma~\ref{add3} on the number of $g$-addable graphs.
This will be the crucial ingredient
in modifying the original proof of Lemma~\ref{uniform},
via the following lemma:

\begin{Lemma} \label{foruniform}
Let $g \geq 0$, $a>1$, $A<3$ and $\eta^{\prime} > 0$ be constants.
Then there exist $n_{1}$ and $\delta_{1} > 0$ such that,
for all $n \geq n_{1}$ and all constants $m_{1}, m_{2}$ satisfying
$an \leq m_{1} < m_{2} \leq An$
and $m_{2} - m_{1} \leq \delta_{1} n + 1$,
we have
\begin{displaymath}
|\mathcal{S}^{g}(n,m_{1})| 
\leq (1 + \eta^{\prime})^{n} |\mathcal{S}^{g}(n,m_{2})|.
\end{displaymath} 
\end{Lemma}
\begin{proof}
By Lemma~\ref{add3},
there exists a constant $c$ such that,
for all sufficiently large $n$,
the (vast) majority of graphs in $\mathcal{S}^{g}(n,m_{2}-1)$
have at least $cn$ $g$-addable edges.
Note that we can obtain a graph in $\mathcal{S}^{g}(n,m_{2})$
by adding such an edge,
and that any such graph will be built at most $m_{2} < 3n$ times.
Thus,
\begin{eqnarray*}
|\mathcal{S}^{g}(n,m_{2})| & \geq &
\frac{cn \frac{1}{2} |\mathcal{S}^{g}(n,m_{2}-1)|}{3n} \\
& = & \frac{c}{6} |\mathcal{S}^{g}(n,m_{2}-1)|.
\end{eqnarray*}

Proceeding in this manner,
we may consequently obtain
\begin{eqnarray*}
|\mathcal{S}^{g}(n,m_{1})| 
& < & \left( \frac{6}{c} \right)^{\delta_{1}n+1} |\mathcal{S}^{g}(n,m_{2})| \\
& = & \left( \left( \frac{6}{c} \right)^{\delta_{1}+\frac{1}{n}} \right)^{n}
|\mathcal{S}^{g}(n,m_{2})|, 
\end{eqnarray*}
from which the result follows.
\end{proof}

We may now obtain our uniform convergence result:

\begin{Theorem} \label{newuniform}
Let $g \geq 0$, $a > 1$, $A < 3$ and $\eta>0$ be constants.
Then there exists $n_{0}$ such that, 
for all $n \geq n_{0}$ and all $m \in [an,An]$,
we have
\begin{displaymath}
\left| \left( \frac{|\mathcal{S}^{g}(n,m)|}{n!} \right)^{1/n} - 
\gamma \left( \frac{m}{n} \right) \right| < \eta.
\end{displaymath} 
\end{Theorem}
\begin{proof}
The proof follows that of Lemma 2.9 in~\cite{ger},
using Lemma~\ref{foruniform} at the relevant point.
\end{proof}

\section{Triangulated appearances: proof of Lemma~\ref{main6apps}}
\label{6appsection}

In this section,
we shall now turn our attention to triangulated appearances
(see Definition~\ref{6apps}).
The main feature here will be the proof of Lemma~\ref{main6apps}.

Throughout this section,
we shall often be interested in the region
when $\frac{m}{n}$ is close to $3$,
and so we start with a lemma that defines the relevant growth constant
$\gamma (3)$:

\begin{Lemma} [\cite{ger}, Theorem 2.1]
There exists a constant $\gamma (3)>0$ such that
\begin{displaymath}
\gamma (q) \to \gamma (3) \textrm{ as } q \to 3^{-},
\end{displaymath}
i.e.~$\gamma (q) \to \gamma (3)$ as $q \to 3$ from below.
\end{Lemma}

We next give two useful lemmas on the size of $\mathcal{S}^{g}(n,m)$
when $\frac{m}{n}$ is close to~$3$.
The first gives a lower bound:

\begin{Lemma} \label{lower}
Let $g \geq 0$ be a constant,
and let $m=m(n)$ satisfy
$\frac{m}{n} \to 3$ as $n \to \infty$.
Then,
given any $\epsilon > 0$,
there exists $N_{1}$ such that 
\begin{displaymath}
|\mathcal{S}^{g}(n,m)| \geq (1- \epsilon)^{n} n! (\gamma (3))^{n}
\textrm{ for all } n \geq N_{1}.
\end{displaymath}
\end{Lemma}
\begin{proof}
The proof is by induction on $g$.

Base case ($g=0$):
This follows from Lemma~\ref{uniform}
and the continuity of $\gamma (q)$ as $q \to 3^{-}$.

Inductive step:
Now suppose the result is true for all $g<k$ and consider $g=k$.
It suffices to deal with the cases when
(a) $m \leq 3n-6+6(k-1)$ for all large $n$
and (b) $m > 3n-6+6(k-1)$ for all large $n$.

Case (a): 
We have 
$|\mathcal{S}^{k}(n,m)| \geq |\mathcal{S}^{k-1}(n,m)| \geq
(1- \epsilon)^{n} n! (\gamma (3))^{n}$
by the induction hypothesis,
and so we are done.

Case (b): 
We shall construct graphs in $\mathcal{S}^{k}(n,m)$
by using triangulations in 
$\mathcal{S}^{k-1}(\lceil \frac{n}{2} \rceil, 
3 \lceil \frac{n}{2} \rceil - 6 + 6(k-1))$
and 
$\mathcal{S}^{0}(\lfloor \frac{n}{2} \rfloor, 
3 \lfloor \frac{n}{2} \rfloor - 6)$:

Take a graph $T_{1}$ in 
$\mathcal{S}^{k-1}(\lceil \frac{n}{2} \rceil, 
3 \lceil \frac{n}{2} \rceil - 6 + 6(k-1))$
(we may assume that we have at least
$(1- \frac{\epsilon}{2})^{\lceil \frac{n}{2} \rceil} 
\lceil \frac{n}{2} \rceil ! 
(\gamma (3))^{\lceil \frac{n}{2} \rceil}$
choices for this,
by the induction hypothesis).
Choose $\lceil \frac{n}{2} \rceil$ vertices from $1$ to $n$
($\left( ^{\phantom{q}n}_{\lceil \frac{n}{2} \rceil} \right)
= \frac{n!}{\lceil \frac{n}{2} \rceil ! \lfloor \frac{n}{2} \rfloor !}$
choices),
and use these to create a copy of $T_{1}$
such that the increasing bijection
from $\{ 1,2, \ldots, \lceil \frac{n}{2} \rceil \}$
to these vertices gives an isomorphism.

Now fix a specific embedding of this graph
on a surface of genus $k-1$,
and select one of the faces to be the `outer' face.
Note that it is also certainly possible to select an `inner' face 
that has at most one vertex in common with this outer face
(since if two triangular faces share two common vertices,
then they must share a common edge,
but an edge can only be in at most two faces
and so this rules out only at most three inner faces)
--- let us call such a vertex `exceptional'.

Similarly,
we may take a graph $T_{2}$ in
$\mathcal{S}^{0}(\lfloor \frac{n}{2} \rfloor, 
3 \lfloor \frac{n}{2} \rfloor - 6)$
(we have at least
$(1- \frac{\epsilon}{2})^{\lfloor \frac{n}{2} \rfloor} 
\lfloor \frac{n}{2} \rfloor ! 
(\gamma (3))^{\lfloor \frac{n}{2} \rfloor}$
choices),
put an isomorphic copy of this on the remaining vertices,
fix a specific embedding of this on the plane,
and find an inner face
that has at most one vertex 
in common with the outer face
(again,
we shall call such a vertex `exceptional').

Let us now join the two graphs by adding six edges 
between the two outer faces, 
thus creating a triangulation in 
$\mathcal{S}^{k-1}(n,3n-6+6(k-1))$
(if the two chosen inner faces
each contain an exceptional vertex,
$v_{1}$ and $v_{2}$ say,
then we should do this in such a way 
that $v_{1}v_{2}$ is not one of the edges added).

We may then also add up to six edges between vertices 
in the two chosen inner faces
to create a (simple) graph in $\mathcal{S}^{k}(n,m)$.
See Figure~\ref{lemmafig}
(where we imagine that a tunnel/handle joins the `holes' at $A$ and $B$).

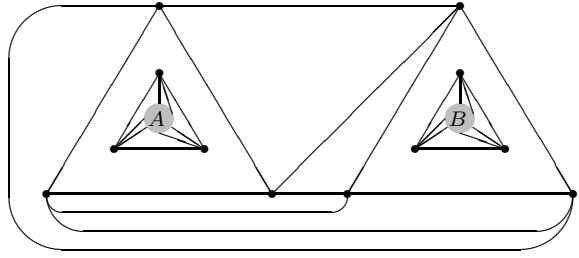
\begin{figure} [ht]
\setlength{\unitlength}{1cm}
\begin{picture}(10,3.25)(-1.75,-0.75)

\put(1.175,0.8){
\begin{tikzpicture}
\fill[lightgray] (0,0) circle (0.2);
\end{tikzpicture}
}

\put(5.175,0.8){
\begin{tikzpicture}
\fill[lightgray] (0,0) circle (0.2);
\end{tikzpicture}
}

\put(0,0){\line(1,0){4}}
\put(1.5,2.5){\line(3,-5){1.5}}
\put(1.5,2.5){\line(-3,-5){1.5}}
\put(0.9,0.6){\line(1,0){1.2}}
\put(1.5,1.6){\line(3,-5){0.6}}
\put(1.5,1.6){\line(-3,-5){0.6}}
\put(0,0){\circle*{0.1}}
\put(3,0){\circle*{0.1}}
\put(1.5,2.5){\circle*{0.1}}
\put(0.9,0.6){\circle*{0.1}}
\put(2.1,0.6){\circle*{0.1}}
\put(1.5,1.6){\circle*{0.1}}

\put(4,0){\line(1,0){3}}
\put(5.5,2.5){\line(3,-5){1.5}}
\put(5.5,2.5){\line(-3,-5){1.5}}
\put(4.9,0.6){\line(1,0){1.2}}
\put(5.5,1.6){\line(3,-5){0.6}}
\put(5.5,1.6){\line(-3,-5){0.6}}
\put(4,0){\circle*{0.1}}
\put(7,0){\circle*{0.1}}
\put(5.5,2.5){\circle*{0.1}}
\put(4.9,0.6){\circle*{0.1}}
\put(6.1,0.6){\circle*{0.1}}
\put(5.5,1.6){\circle*{0.1}}

\put(3,0){\line(1,1){2.5}}
\put(1.5,2.5){\line(1,0){4}}
\put(2,0){\oval(4,0.5)[b]}
\put(3.5,0){\oval(7,1)[b]}
\put(3.25,0){\oval(7.5,1.5)[b]}
\put(1.5,0){\oval(4,5)[tl]}

\put(1.35,0.9){\footnotesize{$A$}}
\put(5.35,0.9){\footnotesize{$B$}}

\put(1.5,1.6){\line(0,-1){0.4}}
\put(5.5,1.6){\line(0,-1){0.4}}

\put(1.5,1.6){\line(1,-3){0.18}}
\put(5.5,1.6){\line(1,-3){0.18}}

\put(2.1,0.6){\line(-3,2){0.43}}
\put(6.1,0.6){\line(-3,2){0.43}}

\put(2.1,0.6){\line(-3,1){0.6}}
\put(6.1,0.6){\line(-3,1){0.6}}

\put(0.9,0.6){\line(3,2){0.43}}
\put(4.9,0.6){\line(3,2){0.43}}

\put(0.9,0.6){\line(1,1){0.4}}
\put(4.9,0.6){\line(1,1){0.4}}
\end{picture}
\caption{
Joining two triangulations to create a graph in $\mathcal{S}^{k}(n,m)$.
} \label{lemmafig}
\end{figure}

Now let us consider the amount of double-counting.
We have at most 
$\left( ^{m}_{\phantom{.}6} \right) \left( ^{m}_{\phantom{.}6} \right) 
\leq (4n)^{12}$
possibilities for the inserted edges,
and then at most $2$ possibilities 
for which triangulation is $T_{1}$ and which is $T_{2}$.
Hence,
each graph is constructed at most $2\cdot4^{12} n^{12}$ times.

Thus,
\begin{eqnarray*}
\phantom{wwwwww}
|\mathcal{S}^{k}(n,m)|
& \geq &
\frac{(1- \frac{\epsilon}{2})^{n} n! (\gamma (3))^{n}}{2\cdot4^{12} n^{12}} \\
& \geq & (1- \epsilon)^{n} n! (\gamma (3))^{n}
\textrm{ for sufficiently large $n$.} 
\phantom{wwwwww}
\qedhere
\end{eqnarray*}
\end{proof}

We now also obtain a useful upper bound:

\begin{Lemma} \label{upper}
Let $g \geq 0$ be a constant.
Then, given any $c>0$ and any $\epsilon > 0$,
there exist $q \in (3-c,3)$ and $N_{2}$ such that 
\begin{displaymath}
|\mathcal{S}^{g}(n,\lfloor qn \rfloor)| 
\leq (1+ \epsilon)^{n} n! (\gamma (3))^{n}
\textrm{ for all } n \geq N_{2}.
\end{displaymath}
\end{Lemma}
\begin{proof}
This follows from the fact that
$\left( \frac{|\mathcal{S}^{g}(n,\lfloor qn \rfloor)|}{n!} \right)^{1/n}
\to \gamma (q)$ for $q \in (1,3)$,
together with the continuity of $\gamma (q)$ as $q \to 3^{-}$.
\end{proof}

Before we proceed with the proof of Lemma~\ref{main6apps},
we shall find it very helpful to first establish bounds
on the number of possible intersections of different triangulated appearances:

\begin{Lemma} \label{intersections}
The total edge set of a triangulated appearance of a connected graph of order $t$
will intersect (i.e.~have an edge in common with)
the total edge set of at most 
$\left( ^{t+3} _{\phantom{w}3} \right)$ other triangulated appearances 
of connected graphs of order~$t$.
\end{Lemma}
\begin{proof}
Suppose we have a triangulated appearance of a connected graph of order $t$ 
at $W \subset V(G)$, 
as in Definition~\ref{6apps} and Figure~\ref{6appfig}.
Note that the vertices $\{ v_{1}, v_{2}, v_{3} \}$ form a $3$-vertex-cut.
Suppose that $G$ also contains another triangulated appearance 
of a connected graph of order $t$,
at $W_{2}$ say,
and let $\{ u_{1}, u_{2}, u_{3} \}$ 
denote the associated $3$-vertex-cut in $V(G) \setminus W_{2}$. 
See Figure~\ref{interfig}.

\begin{figure} [ht]
\setlength{\unitlength}{1cm}
\begin{picture}(10,3.5)(-1,-0.2)

\put(0.375,1){
\begin{tikzpicture}
\fill[lightgray] (0,0) circle (0.5);
\end{tikzpicture}
}

\put(5.375,1){
\begin{tikzpicture}
\fill[lightgray] (0,0) circle (0.5);
\end{tikzpicture}
}

\put(0,0){\line(0,1){3}}
\put(0,0){\line(2,1){3}}
\put(0,3){\line(2,-1){3}}
\put(0,0){\line(1,3){0.5}}
\put(0,0){\line(1,1){1}}
\put(0,3){\line(1,-3){0.5}}
\put(0,3){\line(1,-1){1}}
\put(1,1){\line(4,1){2}}
\put(1,2){\line(4,-1){2}}
\put(0,0){\circle*{0.1}}
\put(0,3){\circle*{0.1}}
\put(3,1.5){\circle*{0.1}}
\put(0.5,1.5){\circle*{0.1}}
\put(1,1){\circle*{0.1}}
\put(1,2){\circle*{0.1}}
\put(0.8,1.4){\large{$W$}}
\put(-0.2,-0.3){$v_{3}$}
\put(-0.2,3.2){$v_{1}$}
\put(3.1,1.4){$v_{2}$}

\put(5,0){\line(0,1){3}}
\put(5,0){\line(2,1){3}}
\put(5,3){\line(2,-1){3}}
\put(5,0){\line(1,3){0.5}}
\put(5,0){\line(1,1){1}}
\put(5,3){\line(1,-3){0.5}}
\put(5,3){\line(1,-1){1}}
\put(6,1){\line(4,1){2}}
\put(6,2){\line(4,-1){2}}
\put(5,0){\circle*{0.1}}
\put(5,3){\circle*{0.1}}
\put(8,1.5){\circle*{0.1}}
\put(5.5,1.5){\circle*{0.1}}
\put(6,1){\circle*{0.1}}
\put(6,2){\circle*{0.1}}
\put(5.8,1.4){\large{$W_{2}$}}
\put(4.8,-0.3){$u_{3}$}
\put(4.8,3.2){$u_{1}$}
\put(8.1,1.4){$u_{2}$}

\end{picture}
\caption{Triangulated appearances at $W$ and $W_{2}$.} \label{interfig}
\end{figure}
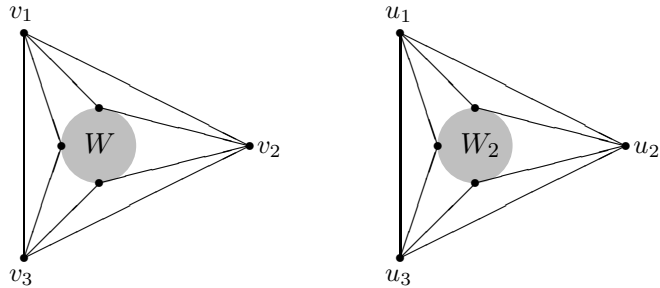

Firstly,
suppose $\{u_{1}, u_{2}, u_{3} \} \subset V(G) \setminus W$.
Note that all vertices in $W$ 
will then be contained within a single component of 
$G \setminus \{u_{1},u_{2},u_{3}\}$,
together with any vertex $v_{i}$ not in $\{ u_{1}, u_{2}, u_{3} \}$.
Thus,
since $W_{2}$ will be one of the components of 
$G \setminus \{ u_{1}, u_{2}, u_{3} \}$,
we must have
either $W_{2} \supset W$,
in which case $W_{2} = W$
(since $|W_{2}|=|W|=t$),
or $W_{2} \subset V(G) \setminus (W \cup \{ v_{1}, v_{2}, v_{3} \})$,
in which case the total edge set of $W_{2}$ 
does not intersect the total edge set of $W$.

Alternatively, 
suppose $u_{i} \in W$ for some $i$.
Then we have
$\{ u_{1}, u_{2}, u_{3} \} \subset (W \cup \{ v_{1}, v_{2}, v_{3} \} )$,
since $u_{1}$, $u_{2}$ and $u_{3}$ are all neighbours of each other,
and so there are then at most 
$\left( ^{t+3} _{\phantom{w}3} \right)$ possibilities 
for $\{ u_{1}, u_{2}, u_{3} \}$.
For each one of these,
let us consider the number of possibilities for $W_{2}$,
which we recall will be one of the components of 
$G \setminus \{ u_{1}, u_{2}, u_{3} \}$.

Since $u_{i} \in W$ and $u_{i} \notin W_{2}$,
we know that $W_{2}$ must contain at least one vertex in $V(G) \setminus W$
in order that $|W_{2}| = t$.
Note also that we can't have 
$W_{2} \subset V(G) \setminus (W \cup \{ v_{1}, v_{2}, v_{3} \} )$,
since two vertices in $W_{2}$ must be adjacent to $u_{i}$
(which is in $W$).
Hence,
$W_{2}$ must contain at least one vertex from $\{ v_{1}, v_{2}, v_{3} \}$.
But no two vertices from $\{ v_{1}, v_{2}, v_{3} \}$
can be in different components of $G \setminus \{ u_{1}, u_{2}, u_{3} \}$,
since $\{ v_{1}, v_{2}, v_{3} \}$ forms a triangle.
Hence,
$W_{2}$ is in fact the unique component of 
$G \setminus \{u_{1}, u_{2}, u_{3} \}$ 
containing vertices from $\{ v_{1}, v_{2}, v_{3} \}$,
and so we find that (given $\{ u_{1}, u_{2}, u_{3} \}$)
we have only one possibility for $W_{2}$.
\end{proof}

We are now ready to obtain the main result of this section
(Lemma~\ref{main6apps}),
showing that whp $\mathcal{S}^{g}_{n,m}$
will have linearly many triangulated appearances 
of any given planar triangulation.
Again,
we also provide a sketch of the proof:

\begin{6sketch}
Using Lemma~\ref{upper},
we find a value $q$ such that
$|\mathcal{S}^{g}(n, \lfloor qn \rfloor)|
\leq (1 + \epsilon)^{n} n! (\gamma (3))^{n}$
for all large $n$
(for a given $\epsilon > 0$).
We then suppose
(aiming for a contradiction)
that the statement of the theorem is false for some $n=k$,
and we recall
$|\mathcal{S}^{g}(k,m(k))| \geq (1 - \epsilon)^{k} k! (\gamma (3))^{k}$
from Lemma~\ref{lower}.

We then take graphs in $|\mathcal{S}^{g}(k,m(k))|$
without $\alpha k$ totally edge-disjoint triangulated appearances of $T$,
and to each of these 
we attach many triangulated appearances 
of carefully selected graphs $T_{1}$ and $T_{2}$,
which both contain triangulated appearances of $T$.
By choosing $T_{1}$ and $T_{2}$
to have the appropriate ratio of edges to vertices,
we may consequently construct many graphs
in $\mathcal{S}^{g} ((1+\delta)k, \lfloor q (1+\delta)k \rfloor)$
for some $\delta > 0$.

The fact that the original graphs 
were assumed to contain few triangulated appearances of $T$
is then used to bound the amount of double-counting,
and so we find that we obtain a contradiction to our earlier upper bound
on $|\mathcal{S}^{g} (n, \lfloor qn \rfloor)|$ when $n=(1+\delta)k$.
\end{6sketch}
\begin{6full}
Due to Lemma~\ref{mainapps}
(taking $H$ to be any graph containing a triangulated appearance of $T$ 
that doesn't involve vertex $1$),
it only remains to deal with the case when $\frac{m}{n} \to 3$.

Let 
\begin{equation}
\beta = e^{2} ( \gamma (3) )^{|T|+4} 
\left( 4 \left(^{|T|+7}_{\phantom{qw}3} \right) +4 \right) (|T|+4)!,
\label{tribeta}
\end{equation}
and let $\alpha$ be a fixed constant in $\left( 0, \frac{1}{\beta} \right)$.
Then we have $\alpha \beta <1$,
and so there exists $\epsilon \in \left( 0, \frac{1}{3} \right)$
such that 
\begin{equation}
(\alpha \beta)^{\alpha} = 1 - 3 \epsilon. \label{triepsilon}
\end{equation}
Given this $\epsilon$,
let $q \in (3 - \frac{\alpha}{1 + \alpha (|T|+4)}, 3)$ 
be as defined by Lemma~\ref{upper}.

Now let us take some large $N$,
and let us suppose that the statement of the theorem doesn't hold 
for some $k \geq N$
(throughout the remainder of this proof,
we will assume that $N$ is large enough that various inequalities involving 
$\beta, \alpha, \epsilon, q$ and $k$ are satisfied).
Let $M$ denote $m(k)$,
and let $\mathcal{G}_{k}$ denote the set of graphs in $\mathcal{S}^{g}(k,M)$
which do not contain
at least $\alpha k$ totally edge-disjoint triangulated appearances of $T$
(so $|\mathcal{G}_{k}| \geq 
e^{- \alpha k} (1 - \epsilon)^{k} (\gamma (3))^{k} k!$).

Let $T_{1}$ be the triangulation produced 
by taking an order-preserving copy of $T$ 
on the vertex set $\{5,6, \ldots, |T|+4\}$
and attaching the vertices $1$, $2$, $3$ and $4$
in the manner shown in Figure~\ref{T1fig}.
Let $T_{2}$ be the graph formed from $T_{1}$ by deleting the edge between $1$ and $4$
(again, see Figure~\ref{T1fig}).
Note that any rooted triangulated appearance of $T_{1}$ or $T_{2}$
will thus contain a triangulated appearance of $T$.

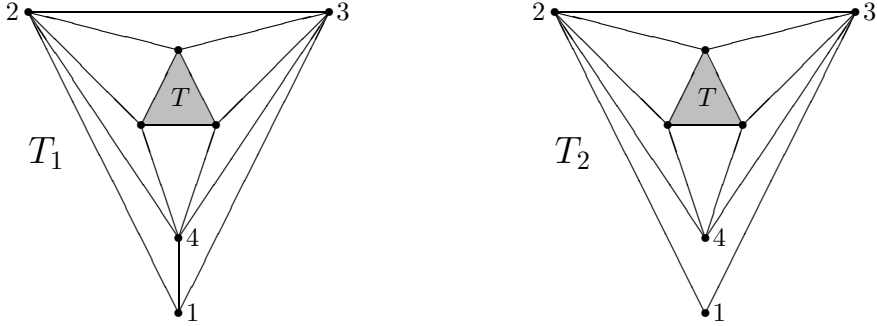
\begin{figure} [ht]
\setlength{\unitlength}{1cm}
\begin{picture}(10,4.1)(0.5,-1)

\put(1.375,1.5){
\begin{tikzpicture}
\fill[lightgray] (1.5,1.5) -- (2.5,1.5) -- (2,2.5) -- cycle;
\end{tikzpicture}
}

\put(8.375,1.5){
\begin{tikzpicture}
\fill[lightgray] (1.5,1.5) -- (2.5,1.5) -- (2,2.5) -- cycle;
\end{tikzpicture}
}

\put(1.5,1.5){\line(1,0){1}}
\put(1.5,1.5){\line(1,2){0.5}}
\put(2.5,1.5){\line(-1,2){0.5}}
\put(1.5,1.5){\circle*{0.1}}
\put(2.5,1.5){\circle*{0.1}}
\put(2,2.5){\circle*{0.1}}

\put(0,3){\line(1,0){4}}
\put(2,0){\line(-2,3){2}}
\put(2,0){\line(2,3){2}}
\put(2,0){\circle*{0.1}}
\put(0,3){\circle*{0.1}}
\put(4,3){\circle*{0.1}}

\put(2,2.5){\line(-4,1){2}}
\put(2,2.5){\line(4,1){2}}
\put(2,0){\line(-1,3){0.5}}
\put(2,0){\line(1,3){0.5}}
\put(0,3){\line(1,-1){1.5}}
\put(4,3){\line(-1,-1){1.5}}

\put(2,-1){\line(-1,2){2}}
\put(2,-1){\line(1,2){2}}
\put(2,-1){\line(0,1){1}}
\put(2,-1){\circle*{0.1}}

\put(2.1,-1.1){$1$}
\put(2.1,-0.1){$4$}
\put(-0.3,2.9){$2$}
\put(4.1,2.9){$3$}
\put(0,1){\LARGE{$T_{1}$}}

\put(1.9,1.75){$T$}

\put(8.5,1.5){\line(1,0){1}}
\put(8.5,1.5){\line(1,2){0.5}}
\put(9.5,1.5){\line(-1,2){0.5}}
\put(8.5,1.5){\circle*{0.1}}
\put(9.5,1.5){\circle*{0.1}}
\put(9,2.5){\circle*{0.1}}

\put(7,3){\line(1,0){4}}
\put(9,0){\line(-2,3){2}}
\put(9,0){\line(2,3){2}}
\put(9,0){\circle*{0.1}}
\put(7,3){\circle*{0.1}}
\put(11,3){\circle*{0.1}}

\put(9,2.5){\line(-4,1){2}}
\put(9,2.5){\line(4,1){2}}
\put(9,0){\line(-1,3){0.5}}
\put(9,0){\line(1,3){0.5}}
\put(7,3){\line(1,-1){1.5}}
\put(11,3){\line(-1,-1){1.5}}

\put(9,-1){\line(-1,2){2}}
\put(9,-1){\line(1,2){2}}
\put(9,-1){\circle*{0.1}}

\put(9.1,-1.1){$1$}
\put(9.1,-0.1){$4$}
\put(6.7,2.9){$2$}
\put(11.1,2.9){$3$}
\put(7,1){\LARGE{$T_{2}$}}

\put(8.9,1.75){$T$}
\end{picture}
\caption{The graphs $T_{1}$ and $T_{2}$.} \label{T1fig}
\end{figure}

Let us use $t$ to denote $|T|+4$
(so $|T_{1}| = |T_{2}| = t$),
and let 
\begin{equation}
\delta = \frac{t \lceil \alpha k \rceil}{k}. \label{tridelta}
\end{equation}
Starting with graphs in $\mathcal{G}_{k}$,
we shall construct graphs 
in $\mathcal{S}^{g}((1+\delta)k, \lfloor q(1+\delta)k \rfloor)$
by attaching 
$k_{1}$ rooted triangulated appearances of $T_{1}$ 
and $k_{2} = \lceil \alpha k \rceil - k_{1}$ 
rooted triangulated appearances of $T_{2}$.
Note that we shall need to achieve the correct balance of $k_{1}$ and $k_{2}$
so that our constructed graphs will indeed have 
$\lfloor q(1+\delta)k \rfloor) 
= \lfloor q ( k + t \lceil \alpha k \rceil ) \rfloor$ edges.

Observe that the total edge set of a triangulated appearance of $T_{1}$ 
will have size $3t$
and the total edge set of a triangulated appearance of $T_{2}$ 
will have size $3t-1$.
Thus,
$k_{1}$ can be chosen so that the number of edges in our constructed graphs
is any desired integer from $M + (3t-1) \lceil \alpha k \rceil$
to $M + 3t \lceil \alpha k \rceil$.

Now recall that
$q > 3 - \frac{\alpha}{1 + \alpha t}$,
and hence that we may assume that $k$ is large enough that
$q > 3 - \frac{ \lceil \alpha k \rceil + 6-6g}{k+t \lceil \alpha k \rceil}$.
Thus,
\begin{eqnarray*}
q (k + t \lceil \alpha k \rceil) & \geq & 
3 (k + t \lceil \alpha k \rceil) - ( \lceil \alpha k \rceil + 6 - 6g ) \\
& = & 3k - 6 + 6g + \lceil \alpha k \rceil (3t-1) \\
& \geq & M + (3t-1) \lceil \alpha k \rceil,
\end{eqnarray*}
and hence
(since the right-hand-side is an integer)
we have 
$\lfloor q ( k + t \lceil \alpha k \rceil ) \rfloor
\geq M + (3t-1) \lceil \alpha k \rceil$.

Also,
since $q<3$ and $\frac{m}{n} \to 3$,
we may assume that $k$ is large enough that
$\lfloor q ( k + t \lceil \alpha k \rceil ) \rfloor
\leq M + 3t \lceil \alpha k \rceil$.
Hence,
we find that we may indeed select a suitable $k_{1}$. \\

Having obtained the appropriate values of $k_{1}$ and $k_{2}$,
let us now construct our graphs 
in $\mathcal{S}^{g}((1+\delta)k, \lfloor q(1+\delta)k \rfloor)$:

Choose $\delta k$ special vertices 
(we have $\left( ^{(1+\delta)k} _{\phantom{qq} \delta k} \right)$ 
choices for these),
and partition them into $\lceil \alpha k \rceil$ unordered blocks of size $t$
(we have 
$\left( ^{\phantom{w}\delta k}_{t, \ldots, t} \right) 
\frac{1}{\left \lceil \alpha k \right \rceil !}$ 
choices for this).
Divide the blocks into two sets of size $k_{1}$ and $k_{2}$.
On each of the first $k_{1}$ blocks,
we put a copy of $T_{1}$ 
such that the increasing bijection from $V(T_{1})$ to the block 
is an isomorphism between $T_{1}$ and this copy.
We do the same for the set of $k_{2}$ blocks, 
except with $T_{2}$ instead of $T_{1}$.

On the remaining (i.e.~non-special) vertices,
choose a graph $G \in \mathcal{G}_{k}$,
and embed $G$ on a surface of genus $g$.
Note that $G$ may be extended to a triangulation by inserting 
$3k-6+6g-M$ `phantom' edges
(such a triangulation may now have multi-edges),
and observe that this triangulation 
will contain $2k-4+4g$ triangles that are faces.
Each of our phantom edges is in exactly two faces of this triangulation,
so when we remove these phantom edges 
we find that our original embedding of $G$ must have contained at least
$2k-4+4g-2(3k-6+6g-M) = 2M-4k+8-8g$
triangles that are faces.

We may attach our copies of $T_{1}$ and $T_{2}$ 
inside $\lceil \alpha k \rceil$ of these triangles
in such a way that we create rooted triangulated appearances 
of $T_{1}$ and $T_{2}$.
See Figure~\ref{6appprooffig}.
Note that we have at least 
$\left( ^{2M-4k+8-8g}_{\phantom{www}\lceil \alpha k \rceil} \right)$
choices for these triangles,
and that we then have $\left \lceil \alpha k \right \rceil !$
choices for which copies of $T_{1}$ and $T_{2}$ to attach within which triangles.

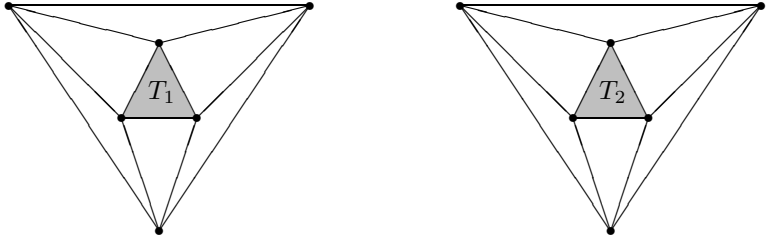
\begin{figure} [ht]
\setlength{\unitlength}{1cm}
\begin{picture}(10,3)(0,0)

\put(1.375,1.5){
\begin{tikzpicture}
\fill[lightgray] (1.5,1.5) -- (2.5,1.5) -- (2,2.5) -- cycle;
\end{tikzpicture}
}

\put(7.375,1.5){
\begin{tikzpicture}
\fill[lightgray] (1.5,1.5) -- (2.5,1.5) -- (2,2.5) -- cycle;
\end{tikzpicture}
}

\put(1.5,1.5){\line(1,0){1}}
\put(1.5,1.5){\line(1,2){0.5}}
\put(2.5,1.5){\line(-1,2){0.5}}
\put(1.5,1.5){\circle*{0.1}}
\put(2.5,1.5){\circle*{0.1}}

\put(2,2.5){\circle*{0.1}}

\put(0,3){\line(1,0){4}}
\put(2,0){\line(-2,3){2}}
\put(2,0){\line(2,3){2}}
\put(2,0){\circle*{0.1}}
\put(0,3){\circle*{0.1}}
\put(4,3){\circle*{0.1}}

\put(2,2.5){\line(-4,1){2}}
\put(2,2.5){\line(4,1){2}}
\put(2,0){\line(-1,3){0.5}}
\put(2,0){\line(1,3){0.5}}
\put(0,3){\line(1,-1){1.5}}
\put(4,3){\line(-1,-1){1.5}}

\put(1.85,1.75){$T_{1}$}

\put(7.5,1.5){\line(1,0){1}}
\put(7.5,1.5){\line(1,2){0.5}}
\put(8.5,1.5){\line(-1,2){0.5}}
\put(7.5,1.5){\circle*{0.1}}
\put(8.5,1.5){\circle*{0.1}}

\put(8,2.5){\circle*{0.1}}

\put(6,3){\line(1,0){4}}
\put(8,0){\line(-2,3){2}}
\put(8,0){\line(2,3){2}}
\put(8,0){\circle*{0.1}}
\put(6,3){\circle*{0.1}}
\put(10,3){\circle*{0.1}}

\put(8,2.5){\line(-4,1){2}}
\put(8,2.5){\line(4,1){2}}
\put(8,0){\line(-1,3){0.5}}
\put(8,0){\line(1,3){0.5}}
\put(6,3){\line(1,-1){1.5}}
\put(10,3){\line(-1,-1){1.5}}

\put(7.85,1.75){$T_{2}$}
\end{picture}
\caption{Creating triangulated appearances of $T_{1}$ and $T_{2}$
inside facial triangles.} \label{6appprooffig}
\end{figure}

Thus, for each choice of special vertices and each choice of $G$, 
the number of graphs in
$\mathcal{S}^{g}((1+\delta)k, \lfloor q(1+\delta)k \rfloor)$
that we may construct is at least
\begin{eqnarray*}
\left( ^{\phantom{w}\delta k}_{t, \ldots, t} \right) 
\left( ^{2M-4k+8-8g}_{\phantom{www}\lceil \alpha k \rceil} \right)
& \geq &
\left( ^{\phantom{w}\delta k}_{t, \ldots, t} \right) 
\left( ^{k+ \lceil \alpha k \rceil}_{\phantom{w}\lceil \alpha k \rceil} \right) 
\textrm{ for large $k$ (since $m/n \to 3$)} \\
& \geq & \left( ^{\phantom{w}\delta k}_{t, \ldots, t} \right)
\frac{k^{\lceil \alpha k \rceil}}{\lceil \alpha k \rceil !} \\
& = & \frac{(\delta k)! k^{\lceil \alpha k \rceil}}
{(t!)^{\lceil \alpha k \rceil} \lceil \alpha k \rceil !} \\
& \geq & \frac{(\delta k)!} {(t! \alpha)^{\lceil \alpha k \rceil}} 
\end{eqnarray*}
(for $k$ large enough that 
$\lceil \alpha k \rceil ! \leq (\alpha k)^{\lceil \alpha k \rceil})$.

Hence,
we may construct at least 
\begin{displaymath}
\left( ^{(1+\delta)k}_{\phantom{qq} \delta k} \right) 
e^{- \alpha k} (1 - \epsilon)^{k} (\gamma(3))^{k} k! 
\frac{(\delta k)!}{(t! \alpha)^{\lceil \alpha k \rceil}}
\end{displaymath}
(not necessarily distinct) graphs 
in $\mathcal{S}^{g}((1+\delta)k, \lfloor q(1+\delta)k \rfloor)$ in total. \\

We shall now consider the amount of double-counting:

Recall that $G$ did not contain 
$\alpha k$ totally edge-disjoint triangulated appearances of $T$,
and that each rooted triangulated appearance of $T_{1}$ 
contains a triangulated appearance of $T$.
Hence, 
$G$ did not contain
$\alpha k$ totally edge-disjoint rooted triangulated appearances of $T_{1}$,
and so (by Lemma~\ref{intersections})
contained fewer than
$\left( \left( ^{t+3} _{\phantom{w}3} \right) +1 \right) \alpha k$  
rooted triangulated appearances of $T_{1}$ in total.

When we deliberately attach a triangulated appearance of $T_{1}$ or $T_{2}$,
the number of `accidental' rooted triangulated appearances of $T_{1}$ 
that we create in the graph
will be at most $\left( ^{t+3} _{\phantom{w}3} \right)$, 
again using Lemma~\ref{intersections},
and so the number of rooted triangulated appearances of $T_{1}$ 
will increase by at most 
$\left( ^{t+3} _{\phantom{w}3} \right) + 1$ each time.
Thus, our created graph will have at most
$\left( \left( ^{t+3} _{\phantom{w}3} \right) +1 \right) \alpha k + 
\left( \left( ^{t+3} _{\phantom{w}3} \right) + 1 \right) \lceil \alpha k \rceil
\leq \left( 2 \left( ^{t+3} _{\phantom{w}3} \right) 
+ 2 \right) \lceil \alpha k \rceil$
rooted triangulated appearances of $T_{1}$.

Similarly,
our created graph will have at most
$\left( 2 \left( ^{t+3} _{\phantom{w}3} \right) 
+ 2 \right) \lceil \alpha k \rceil$
rooted triangulated appearances of $T_{2}$.

Let $x=4\left( ^{t+3} _{\phantom{w}3} \right) + 4$.
Then, given one of our constructed graphs, we have at most
$\left( ^{x \lceil \alpha k \rceil}_{\phantom{i}\lceil \alpha k \rceil} \right)
\leq (xe)^{\lceil \alpha k \rceil}$
choices for which were the special vertices.
Once we have identified these, we then know what $G$ was.
Thus, each graph is constructed at most $(xe)^{\lceil \alpha k \rceil}$ times. \\

Therefore, 
we find that
the number of distinct graphs that we have created in
$\mathcal{S}^{g}((1+\delta)k, \lfloor q(1+\delta)k \rfloor)$
is at least 
\begin{eqnarray*}
& & \left( ^{(1+\delta)k}_{\phantom{qq} \delta k} \right) 
e^{- \alpha k} (1 - \epsilon)^{k} (\gamma(3))^{k} k! 
\frac{(\delta k)!}{(t! \alpha)^{\lceil \alpha k \rceil}} 
(xe)^{- \lceil \alpha k \rceil} \\
& \stackrel{\eqref{tridelta}}{\geq} & 
((1+\delta)k)! (\gamma(3))^{(1+\delta)k} (1-\epsilon)^{k} 
\left( e^{2} (\gamma(3))^{t}xt! \alpha \right) ^{- \lceil \alpha k \rceil} \\
& \stackrel{\eqref{tribeta}}{\geq} & 
((1+\delta)k)! (\gamma(3))^{(1+\delta)k} (1-\epsilon)^{k} 
(\alpha \beta)^{- \lceil \alpha k \rceil} \\
& \geq & |\mathcal{S}^{g}((1+\delta)k, \lfloor q(1+\delta)k \rfloor)| 
(1+\epsilon)^{-(1+\delta)k} (1-\epsilon)^{k} (1-3\epsilon)^{-k} \\
& & \textrm{(by Lemma~\ref{upper} and~\eqref{triepsilon}}) \\
& \geq & 
|\mathcal{S}^{g}((1+\delta)k, \lfloor q(1+\delta)k \rfloor)| 
\left( \frac{(1-\epsilon)}{(1-3\epsilon)(1+\epsilon)^{2}} \right) ^{k} \\
& & \textrm{(since we may assume $k$ is large enough that } \delta<1) \\
& > & |\mathcal{S}^{g}((1+\delta)k, \lfloor q(1+\delta)k \rfloor)| \\
& & \textrm{(since } 
(1-3\epsilon)(1+\epsilon)^{2}=1-\epsilon -5\epsilon^{2} - 3\epsilon^{3}).
\end{eqnarray*}
Thus,
we have obtained our desired contradiction. \qed 
\end{6full}

As an interesting corollary
(by making appropriate choices for $T$),
we also obtain the following new result:

\begin{Corollary}
Let $g \geq 0$ be a constant,
and let $m=m(n)$ satisfy
$\liminf \frac{m}{n} > 1$.
Then,
given any constant $k \geq 3$,
there exist $\alpha > 0$ and $N$ such that
\begin{displaymath}
\mathbb{P}[S_{g}(n,m) 
\textrm{ will have at least $\alpha n$ vertices of degree $k$}] 
> 1 - e^{- \alpha n} 
\textrm{ for all } n \geq N.
\end{displaymath}
\end{Corollary}
\qed

\section{Discussion} \label{discussion}

In this section,
we shall now discuss some of the more interesting unresolved issues.

Two of the most intriguing questions involve the topic of subgraphs.
We see from Table~\ref{subtab} that we have left open the case
when $H$ is a planar multicyclic graph and $\frac{m}{n} \to 1$.
This is in fact an open problem even for $g=0$,
although some results are known if $\frac{m}{n}$ converges to $1$ slowly
(see Theorems 68 and 70 of~\cite{dow}
for the planar case,
the proofs of which actually generalise to any $g$).
The case when $H$ is non-planar and $\liminf \frac{m}{n} \geq 1$
also remains unresolved.
Is it always true that the probability of having such a subgraph 
converges to $0$?

In all of our results,
we note that the value of $g$ seems to have little impact.
It would be interesting to know how this would change
if we were to allow $g$ to grow with $n$,
rather than just being a fixed constant.
Certainly, we would obtain very different behaviour if $g \geq m(n)$,
since then our random graph $S_{g}(n,m)$ would be the same as the standard
Erd\H{o}s-R\'enyi random graph $G(n,m)$.
It would consequently be useful to know more about the typical genus of $G(n,m)$.

\section*{Acknowledgements}

We are very grateful to the referees for their comments.

\end{document}